%% file: log-entropy_manifold_arxiv.tex
\documentclass{article}[10pt]
\input{config.tex}

\def\vep{\varepsilon}
\def\de{\delta}
\def\rife#1{\eqref{#1}}
\def\cP{{\mathcal P}}
\def\into{\int_ M }
\def\be{\begin{equation}}
\def\ee{\end{equation}}
\def\vfi{\varphi}

\def\cH{{\mathcal H}}
\def\cF{\mathcal{F}}
\def\iinto{\int_0^T\!\!\!\!\into}

\numberwithin{equation}{section}

\title{Entropy-minimizing dynamical  transport on Riemannian manifolds}

\author{Gabriele Bocchi \thanks{Department of Mathematics, University of Rome Tor Vergata. Via della Ricerca Scientifica 1, 00133 Roma, Italy. \texttt{bocchi@mat.uniroma2.it}}  \and Alessio Porretta \thanks{Department of Mathematics, University of Rome Tor Vergata. Via della Ricerca Scientifica 1, 00133 Roma, Italy. \texttt{porretta@mat.uniroma2.it}}}

\date{\today}

\begin{document}

\maketitle

\begin{abstract}
Given a smooth  Riemannian manifold $(M,g)$, compact and without boundary,  we analyze the dynamical optimal mass transport problem where the cost is given by the sum of the kinetic energy and the relative entropy with respect to a reference volume measure $e^{-V}dx$. Under the only assumption  that the prescribed marginals  lie in $L^1(M)$, and a lower bound on the Ricci curvature, we characterize the minimal curves as unique weak solutions of the optimality system coupling the continuity equation with a backward Hamilton-Jacobi equation (with source given by $\log (m)$). We give evidence that the entropic cost enhances diffusive effects in the evolution of the optimal densities, proving $L^1\to L^\infty$ regularization in time for any initial-terminal data,  and smoothness of the solutions whenever the marginals are positive and smooth. We use displacement convexity arguments (in the Eulerian approach) and gradient bounds from quasilinear elliptic equations.  We also prove  the convergence of optimal curves  towards the classical Wasserstein geodesics,  as the entropic term is multiplied by a vanishing parameter, showing that  this kind of functionals can be used to build a smoothing approximation of  the standard optimal transport problem.
\end{abstract}

\section{Introduction}

Let $( M , g)$ be a smooth, connected  $d$-dimensional Riemannian manifold, assumed to be compact and  without boundary, endowed with a metric tensor $g=(g_{ij})$ and a volume form $dx$.  We denote by $\cP(M)$ the set of probability measures on $M$. Given a time horizon $T>0$ and two  fixed (initial and terminal) measures $m_0,m_1\in \cP(M)$,  we analyze in this note the optimal transport problem
 \begin{equation}\label{func}
\begin{aligned}
       &  \min\,\,  \mathcal{F}_\epsilon(m,v) \coloneqq\iinto  \frac{1}{2}\abs{v}^2\,dm+ \vep  \int_0^T \cH(m(t) ; \nu) \,, \qquad  \\ 
        & \qquad \hbox{among all }\quad (m,v)\quad \hbox{} :\begin{cases}\partial_t m -div_g(vm)=0\\m(0)=m_0\,, \,\,m(T)=m_1
        \end{cases}
        \end{aligned}
    \end{equation}
where   $\nu:= e^{-V(x)}dx$ and 
$$
\cH(m ; \nu)= \into \log\left(\frac{dm}{ d\nu} \right)dm= \iinto m(\log m+V)\,dxdt
$$
denotes the relative entropy of $m$ with respect to a reference  measure $e^{-V}dx$, given for some Lipschitz continuous vector field $V$. 

In \rife{func}, $m(t)$ is an (absolutely continuous) arc joining $m_0$ and $m_1$ with velocity $v$, $\abs\cdot$ is the length of vector fields and $div_g(\cdot)$ the intrinsic divergence operator on the manifold $M$.
%
\vskip0.4em
The functional \rife{func} can be seen as a perturbation of the
kinetic energy functional used in the dynamical version of  mass transportation \cite{BB}. The additional  term  in \rife{func} prevents concentration effects by penalizing the relative entropy  and  is supposed to enhance some form of dissipation  along optimal curves. This is only one, yet very natural,  among possibly different entropic regularizations of the classical optimal transport energy.  In this respect, it follows a stream of research which has been very intensive in recent times, where other kind of regularizations of the Wasserstein distance  were suggested 
 (see \cite{entropic}, \cite{GiTa}, \cite{leonard}, \cite{LMS}).

The evolution of optimal transport densities with additional costs that consider the effects of  congestion  has been exploited so far in several directions, see e.g.   \cite{BCS}, \cite{BJO},  and especially \cite{LaSa}, where some $L^1-L^\infty$ regularization in the time evolution of the optimal curves was  proved using variational techniques. Similar problems were addressed  in  \cite{Carda1}, \cite{CG}, \cite{CGPT}, \cite{CMS}, \cite{GMST}, \cite{ORRIERI20191868} with a different approach based on ideas coming from mean field game theory and PDE estimates on the optimality system (state-adjoint state) associated to \rife{func}, which is 
\be\label{opsys}
\begin{cases}
    -\partial_t u+\frac{1}{2}\abs{\nabla u}^2=\epsilon (\log(m)+V(x) )\quad &\text{in $(0,T)\times M $}\\
    \partial_t m-div_g(m\nabla u)=0 &\text{in $(0,T)\times M $}\\
    m(0)=m_0, \quad m(T)=m_1 &\text{in $ M $.}
\end{cases}
\ee 
 As firstly observed by P.-L. Lions \cite{L-college},   \rife{opsys} is just one instance of PDE system appearing in mean field game theory (\cite{LL1}, \cite{LL-japan}) 
  and some smoothness on the optimal curves of  this kind of functionals can be derived 
from gradient estimates on the adjoint state $u$ (the so-called Kantorovich potential, in mass transportation  language). This approach, relying on the ellipticity hidden in the optimaliy system, was thoroughly developed in \cite{munoz1}, \cite{Mu},   \cite{Po}. 
In particular, the   case of functional \rife{func}, with the additional entropy term, was   addressed in \cite{Po} for convex domains of $\R^d$ (with no-flux condition at the boundary) assuming that the marginals $m_0, m_1$ are positive and smooth, in which case the minima can be proved to be positive and smooth for all times.
Similar results were also proved for Gaussian-like measures in the whole Euclidean space.
\vskip0.4em


The goal of this paper is twofold. First of all, we give some general result on problem \rife{func}, under the only assumption that the marginals $m_0,m_1 \in L^1(M)$.
In particular, the marginals do not need to be positive (nor smooth), extending many results of \cite{Po} to the case of nonnegative initial-terminal data. Except for some special results obtained in one dimension \cite{CMP}, the case of compactly supported marginals had not been developed, so far. 

Secondly, we analyze the problem in the setting of a Riemannian manifold, in order to get a more exhaustive comprehension of some crucial tools. In fact, in the genuine optimal transport viewpoint, it is well understood (\cite{Erasquin-McCann}, \cite{Daneri_2008}, \cite{FV}, \cite{LV}, \cite{McCann2}, \cite{Renesse}, \cite{Villani-oldnew})   that the Riemannian setting is the most natural to observe the role of  Ricci curvature in regularity arguments related to displacement convexity and entropy dissipation.  In our results, we will require  the only condition that the Ricci curvature is bounded below, and we will obtain estimates which are totally consistent with the pure mass transport problem, embedding therefore the case  $\vep=0$ in \rife{func} into a family of similar problems. As an example, if $Ric(M)\geq \Lambda I$, we will prove the diplacement $\Lambda-$ convexity of the entropy along Wasserstein geodesics as the limit of $\Lambda_\vep$-convexity along the optimal curves of \rife{func}.


Concerning the characterization of minimal curves of \rife{func}, we summarize our main results in the following statement.

\begin{theorem}\label{main} Let $(M, g)$ be a smooth compact Riemannian manifold without boundary, with $Ric_g(M)$ bounded below.  Let $m_0, m_1 \in L^1( M )\cap\mathcal{P}( M )$, and assume that $V\in W^{2,\infty}( M )$, $\vep >0$. 
Then the functional $\mathcal{F}_\epsilon$ in \rife{func} admits a unique minimum, given by $(m,\nabla u)$,  where $(m,u)$ is  the unique weak solution (in the sense of Definition \ref{def:weak_sol}) of  the system   \eqref{opsys}, with $\int_ M  u(T)m_1=0$.
Moreover we have:
\begin{itemize}
\item[(i)] $m>0$ a.e. in $(0,T)\times M$.

\item[(ii)] $u,m\in L^\infty_{loc}((0,T)\times M)$ and $u(0)\in L^1(dm_0), u(T)\in L^1(dm_1)$.

\item[(iii)] if $m_0, m_1\in W^{1,\infty}(M)$ and are (strictly) positive, and if $V\in C^{k,\alpha}(M)$, then $u \in C^{k+1,\alpha}((0,T)\times M), m\in C^{k,\alpha}((0,T)\times M)$.  
\end{itemize}
\end{theorem}

The statement of Theorem \ref{main} summarizes several different results that we establish later. In fact, we will start from the case of smooth and positive marginals (see Theorem \ref{smoothex}), proving the existence of smooth solutions to the system \eqref{opsys}. To this purpose, we follow the strategy suggested by P.-L. Lions \cite{L-college}, and developed in \cite{Mu}, \cite{Po}, which consists in rewriting the system \rife{opsys} as a quasilinear elliptic equation for $u$ (see \rife{elliptic:long}) and using the continuity method, relying on gradient bounds, in order to produce a smooth solution $u$. Following \cite{Po}, this strategy is first employed for a penalized auxiliary problem \rife{mfg-de}, where the $L^\infty-$ norm of $u$ is readily controlled.  Then a compactness argument yields a smooth solution $(m,u)$ after a suitable normalization of $u$.  

Once we have built smooth solutions of \eqref{opsys}, we will obtain all relevant estimates which remain robust for merely $L^1$ marginals $m_0, m_1$.  
This step includes a bound on the minimal value of $\cF_\vep$, obtained by building suitable competitors, which in turn will yield local uniform bounds on  $u$. 

From the local bounds on $u$ we will also obtain the local $L^\infty$- bound on $m$.  Notice that this $L^1-L^\infty$ regularization on the density is not a straightforward extension from the euclidean case, because the Ricci curvature  is allowed to be negative. In particular, the $L^\infty$ control on $m$ does not follow directly from the displacement convexity inequalities as in \cite{Po}. 

Finally,  we will  derive local (in time) bounds for $Du$ in $L^2$ (from the HJ equation) and for the Fisher information of $m$ (by displacement convexity estimates)  that yield the suitable compactness arguments. This latter step   includes a relaxation result on the system and the convergence to weak solutions, providing with the final characterization of the minima of \rife{func} (see Theorem \ref{weak:thm}). 

Many of those tools, involving estimates and stability on the optimality system \eqref{opsys}, are stable as $\vep \to 0$.  

Indeed,  we conclude the article by giving a result of convergence of the optimal curves towards the Wasserstein geodesic, as $\vep \to0$, as well as  the convergence of $\min \mathcal{F}_\vep $ towards $ \min \mathcal{F}_0$.   This convergence occurs with a rate $O(\vep)$ when the marginals have finite entropy, while we cannot prove a rate better than $O(\vep |\log\vep|)$ for the general case of marginals only in $L^1(M)$.

\begin{theorem}\label{main2} Under the assumptions of Theorem \ref{main}, let $(m_\vep, \nabla u_\vep)$ be the minima of \rife{func}, where $(m_\vep, u_\vep)$ solves   \eqref{opsys}, with $\int_ M  u_\vep (T)m_1=0$, and let $(m,\nabla u)$ be the Wasserstein geodesic between $m_0, m_1$. 
Then, as $\vep \to 0$, we have
\begin{align*}
m_\vep & \to m \quad \hbox{in $C^0([0,T],\cP(M))$ and  weakly in $L^1((0,T)\times M)$,}
\\ 
m_\vep \nabla u_\vep & \to m\nabla u\quad \hbox{weakly in $L^1((0,T)\times M)$,}
\end{align*}
and $\min \mathcal{F}_\epsilon\to \min \mathcal{F}_0$. In particular we have 
$$
\min \mathcal{F}_\epsilon=  \min \mathcal{F}_0 + r_\vep
$$
where $r_\vep= O(\vep)$ if $m_0, m_1$ have finite entropy, otherwise $r_\vep= O(\vep|\log \vep|)$.
\end{theorem}

Last but not least, we will show (see  Theorem \ref{conv_vep2}) that using a suitable approximation of $m_0, m_1$ with   sequences $m_{0\vep}, m_{1\vep}$ of smooth positive functions, the Wasserstein geodesic between $m_0, m_1$ can be approximated by smooth minimizers of $\cF_\vep$, in a way that  $u_\vep$ remains uniformly bounded in Lipschitz norm. Hence, in particular,  the Kantorovich potentials converge uniformly in $M$. 

This latter result  shows that the purpose of smoothing the Wasserstein geodesic can be fully accomplished with the functional \rife{func}; in particular, this gives a general Eulerian strategy  towards the proof of displacement convexity properties of the geodesics of optimal transport. As an example, we recover some results of \cite{Erasquin-McCann},  \cite{Daneri_2008} with an alternative proof, avoiding the use of EVI inequalities in favor of a standard Eulerian approach which is now justified going through problems \rife{func}.

\section{Notations and setting of the problem}\label{notations}
In the following, we recall some elements of Riemannian geometry (see e.g. \cite{taylor2010partial}). Throughout  the paper, $( M , g)$ denotes  a smooth, compact, connected and oriented $d$-dimensional Riemannian manifold without boundary, with metric tensor $g=(g_{ij})$, inverse $g^{-1}=(g^{ij})$ and determinant $\abs{g}$. The orientation induces a unitary volume form $dx$. 
If $w$ and $v$ are two vector fields on $ M $, we denote by 
\begin{equation*}
    w \scalg w \coloneqq\sum_{ij} g_{ij}(x)w_{i}v_j
\end{equation*}
their scalar product in the tangent space $T_x\! M $. The length of a vector field is given by $|w|= \sqrt{w \scalg w }$. Correspondingly, there is a scalar product in the cotangent space $T^*_x\!M$, which is defined on differential $1$-forms $\omega$ and $\nu$ on $M$ as $    \omega \scalg \nu \coloneqq\sum_{ij} g^{ij}(x)\omega_{i}\nu_j$.  

Let $x_j$, $j=1,\dots, d$, be a local system of coordinates: if $u\in C^1( M )$, the covariant gradient of $u$, denoted by $\nabla u$, is the vector field with coordinates $\nabla_i u=g^{ij}(x)u_{x_j}$. Therefore, given $u,v\in C^1( M )$,
\begin{equation*}
    \nabla u\scalg \nabla v=\sum_{ij}g^{ij}(x)u_{x_i}v_{x_j}
\end{equation*}
We denote the Levi-Civita connection associated to the metric $g$ with the letter $D$ and we will derivate covariantly vector and tensor fields on $M$.
Recalling that, in local coordinates, the Christoffel symbols are
\begin{equation*}
    \Gamma^k_{ij}= \frac{1}{2}\sum_l\left(\frac{\partial g_{jl}}{\partial x_i}+\frac{\partial g_{li}}{\partial x_j}-\frac{\partial g_{ij}}{\partial x_l}\right)g^{lk}\,,
\end{equation*}
the covariant derivative of a $C^1$ vector field $X=(X_j)$ along the vector field $v=(v_i)$ is the vector field $D_vX$ with $k$-th coordinate given by
$$
(D_vX)_k=\sum_{ij}v_iX_j\Gamma^k_{ij}+(\nabla X_k)\scalg v.
$$
If $X=(X_j)$ is a $C^1$ vector field on $M$, the divergence of $X$ is defined by
$$
div_gX=\frac{1}{\sqrt{\abs{g}}}\sum_k(\sqrt{\abs{g}}X^k)_{x_k}
$$
 and the Leibniz rule: $div_g(fX)= \nabla f \scalg X+ f div_g X$ holds for every $f\in C^1(M)$ and any $C^1$ vector field $X$ on $M$. Furthermore, by the Stokes theorem, we have
$$
\int_M div_g X\,dx=0\,.
$$
The Hessian $\nabla^2u$ of a $C^2$ function $u$ is the symmetric $2$-tensor given by
$$
(\nabla^2u)(v,w)\coloneqq (D_v\nabla u)\scalg w=(D_w\nabla u)\scalg v
$$
for every vector fields $v,w$ on $M$, where the last equality follows by the symmetry and the compatibility with the metric of the Levi-Civita connection.
The components of the Hessian are the second covariant derivatives, given by
\begin{equation*}
   \nabla_{ij}u=u_{x_ix_j}-\sum_k\Gamma^k_{ij}u_{x_k}.
\end{equation*}
In particular for every $C^2$ functions $f,u$ and every vector field $v$ it holds
$$
v\scalg\nabla(\nabla f\scalg \nabla u)=(\nabla^2 f)(v,\nabla u)+(\nabla^2 u)(v,\nabla f).
$$
As usual, we denote {$\Delta_gu=div_g(\nabla u)$}   the Laplace-Beltrami operator on $ M $. We recall the B\"ochner formula (see e.g. \cite{bochner}):
\be\label{boch}
\tfrac{1}{2}\Delta_g\abs{\nabla f}^2=\abs{\nabla^2f}^2+\nabla (\Delta_gf)\scalg \nabla f+Ricc_g(\nabla f, \nabla f)
\ee
for every $f\in C^3(M)$, where $Ricc_g$ is the curvature tensor of the metric $g$. 
Throughout all the paper,  we will assume that $Ric_g(M)$ is bounded below,  i.e. \be\label{ricci}
Ricc_g(X,X)\ge -\lambda \abs{X}^2
\ee 
for every vector field $X$ on $M$,  for some $\lambda\geq 0$.
  \vskip0.5em
Given $M$, we denote by $\cP(M)$ the space of probability measures on $M$, endowed with the Wasserstein metric.   The characterization of the Wasserstein geodesics in terms of optimal mass transportation on $M$ is well established, since $M$ is compact, see \cite{McCann2}, \cite{Villani-oldnew}.   
We will always identify the measures with their densities when they are  absolutely continuous with respect to the unitary volume measure $dx$. With $L^p(M)$ we  indicate the standard Lebesgue space, for $p\in[1,+\infty]$, and with $W^{k,p}(M)$ the Sobolev space of functions with $k$ weak $L^p$-derivatives.  We will use the Sobolev and Poincar\'e-Wirtinger inequalities on $M$, for which we refer to \cite{hebey}. 
Finally, throughout the paper we denote $Q_T$ the cylinder $(0,T)\times M$, and $\overline Q_T= [0,T]\times M$.

\subsection{Optimal transport  functional} 

We now make precise the sense of the minimization problem \rife{func}. 

\begin{definition}\label{def-conteq-sol} Let $m_0, m_1\in \cP( M )$.  A couple $(m,v)$ is a solution of the continuity equation 
\be\label{conteq}
\begin{cases}
\partial_t m-div_g(vm)=0\\m(0)=m_0\,, \,\,m(T)=m_1\,,
\end{cases}
\ee
if $m\in C([0,T];\cP( M ))$ with $m(0)=m_0$ and $m(T)=m_1$,  $v(t,x)$ is a measurable vector field on $Q_T $ such that $\iinto |v|^2\, dm <\infty$ and the following equality holds
$$
\into \vfi(t)dm(t)-\into \vfi(s)dm(s) + \int_s^t\!\!\!\! \into \left( -\partial_t \vfi + v\scalg \nabla \vfi\right) \, dm =0 \,, 
$$
for every $0\leq s<t\leq T$ and every   function $\vfi\in C^1(\overline Q_T )$.
\end{definition}

We recall (see \cite{AGS}) that  weak solutions as defined above are essentially equivalent to absolutely continuous curves from $[0,T]$ into $\cP( M )$ which have $L^2$ metric derivative.  We also recall that any convex, superlinear function $F(r)$ induces a lower semicontinuous functional on the space of probability measures:
$$
F(m):= \begin{cases} \into F(m)\, dx  & \hbox{if $m$ is absolutely continuous} \\ +\infty  & \hbox{otherwise.}
\end{cases} 
$$
Similar kind of functionals have been extensively studied,  see e.g. \cite{LaSa} and references therein. Even if we could  consider general functions $F$, for the sake of clarity we  restrict  the analysis in this paper  to the specific entropic case, in which $F(m)= m\log(m)$, and more generally to the relative entropy in terms of a possibly inhomogeneous reference measure $\nu= e^{-V(x)}dx$:
\be\label{rel-entr}
\cH(m; \nu):= \into F\left(\frac{dm}{ d\nu} \right)d\nu= \into \log\left(\frac{dm}{ d\nu} \right)dm= \into m(\log m + V)dx \,
\ee
with the convention that $\cH(m; \nu)=+\infty$ whenever $m$ is not absolutely continuous with respect to $dx$.  In what follows, we assume that $V$ is (at least) Lipschitz continuous on $M$.
\vskip0.4em
Thanks to Definition \ref{def-conteq-sol}, the meaning of the optimal transport problem \rife{func} is now clarified, to be read as
\be\label{func-precise}
\begin{aligned}
       &  \min \mathcal{F}_\epsilon(m,v) \coloneqq\iinto \frac{1}{2}\abs{v}^2\,\,dm+ \vep  \int_0^T \cH(m ; \nu) \,, \qquad \nu:= e^{-V(x)}dx\\ 
        & \qquad \hbox{among all }\quad (m,v)\, : \quad \begin{cases}\partial_t m-div_g(vm)=0\\m(0)=m_0\,, \,\,m(T)=m_1
        \end{cases},
    \end{aligned}        
\ee
where the equation is understood as above.

We first establish that, for every   $m_0, m_1\in \cP(M)$, there exists an arc along which  the above functional is finite, so it admits a finite minimum. In addition, we can give a universal upper bound  on the minimal value of $\mathcal{F}_\epsilon$.

\begin{proposition}
\label{prop:exist_comp}
Let \rife{ricci} hold  true.  There exists a constant $C( M,d, \lambda, T, \|V\|_\infty)$ (depending on $M$ as well as on $d,\lambda, T, \|V\|_\infty$) such that
    \begin{equation}\label{upper-min}
        \min \mathcal{F}_\epsilon\leq C( M,d, \lambda, T, \|V\|_\infty)
    \end{equation}
    for every $m_0,m_1\in\cP( M )$, and every $\vep\leq 1$.
\end{proposition}
\begin{proof}

    Consider the heat kernel $p_t(x,y)$ associated to the volume measure $dx$ and the curve $\mu_0(\cdot):[0,1]\to\cP( M )\cap C^\infty_+( M )$  generated by the heat semigroup $S_t$:
    $$
    t\to \mu_0(t,x)= S_t (m_0)\coloneqq\int_ M  p_t(x,y)\,dm_0(y).
    $$
It is a classical result (cfr. \cite[Chapters 7, 8]{grigoryan2009heat}) that   $\mu_0(\cdot)$ is well defined and is  a smooth solution of the heat equation
    $$
    \frac{\partial}{\partial t}\mu_0=\Delta_g \mu_0\qquad \hbox{on $[0,\infty)\times M$.}
    $$ 
    In particular we have $\mu_0(t,\cdot)>0$ for every $t>0$ by the strong maximum principle. 
    It follows that the velocity of such curve for $t>0$ is given by the vector field
    $$
        \nu_0(t,x)\coloneqq\frac{\nabla \mu_0}{\mu_0} \,.
    $$
    By the Li-Yau inequality \cite[Theorem 1.4]{li_yau} we know that there exists a constant $C(d,\lambda)$ such that  
    $$
    \frac{\abs{\nabla \mu_0}^2}{\mu_0 }-2\frac{\partial}{\partial t}\mu_0\leq 
    (C+\frac{2d}{t})\,  \mu_0\,.
    $$
    Recalling that $\mu_0$ is a probability density for every $t>0$, integrating the above inequality we get 
    \begin{equation}
    \label{bound:heat_comp:vel}    
       \int_ M  \abs{\nu_0}^2\mu_0\,dx\le C+\frac{2d}{t}. 
    \end{equation}
    We now study the decay of the entropy along the heat flow. We recall that 
    $$
     \frac{\partial }{\partial t}\int_ M  \mu_0\log(\mu_0) \,dx= \frac{\partial }{\partial t}\int_ M  (\mu_0\log(\mu_0)-\mu_0)\,dx=- \int_ M  \frac{\abs{\nabla \mu_0}^2}{\mu_0}\,dx\,.
    $$
    By Sobolev and Poincaré-Wirtinger inequality we have (for $2^*=\frac{2d}{d-2}$ if $d>2$, or $2^*$ any sufficient large number if $d=2$) 
\begin{align*}
   \int_ M \frac{\abs{\nabla \mu_0}^2}{\mu_0}\,dx&=  4\int_ M \abs{\nabla \sqrt{\mu_0}}^2\,dx \ge C_S\left(\int_ M \abs{\sqrt{\mu_0}-Vol( M )^{-1}\!\!\int_ M \sqrt{\mu_0}\,dx}^{2^*}dx\right)^{\frac{2}{2^*}}\\
    &\ge c_1(\int_ M \sqrt{\mu_0}^{2^*}dx)^\frac{2}{2^*}-c_2 (\int_ M \sqrt{\mu_0}\,dx)^2\\
    &\ge c_1(\int_ M \sqrt{\mu_0}^{2^*}dx)^\frac{2}{2^*}-c_2  Vol( M )\,.
\end{align*}
By the concavity of the $\log$ function and Jensen inequality for the probability measure $\mu_0$
\be\label{fish}
\begin{split}
     \log\left(\int_ M \frac{1}{\mu_0}\abs{\nabla \mu_0}^2\,dx+c_2 Vol( M )\right)&\ge \frac{2}{2^*}\log\left(\int_ M \sqrt{\mu_0}^{2^*}dx\right)+\log(c_1)\\
    &\ge \frac{2}{2^*}\int_ M \log\left( \mu_0^{\frac{2^*-2}2}\right)\mu_0dx+\log(c_1)\\
    &= \frac{2}{d}\int_ M    \log\left(\mu_0\right)\mu_0dx+\log(c_1)\,.
\end{split}
\ee
In other words, if $\varphi(t)\coloneqq\int_ M  \mu_0\log \mu_0 \,dx$, then we deduce
    $$
    \varphi'(t)\le- c_1e^{\frac{2}{d}\varphi}+C
    $$
    for a constant $C$ depending only on $Vol( M )$ and $d$. 
    This implies
    $$
    (\varphi'(t)-C)e^{-\frac{2}{d}(\varphi-Ct)} \leq- c_1 e^{\frac{2C}{d}t} \leq - c_1 
    $$
 and then,  integrating in $(t_0,t_1)$, we get
    $$
    -\frac{d}{2}e^{-\frac{2}{d}(\varphi(t_1)-Ct_1)}+\frac{d}{2}e^{-\frac{2}{d}(\varphi(t_0)-Ct_0)}+c_1(t_1-t_0)\leq0\,.
    $$
In particular, letting $t_0\to 0$ we deduce
    $$
    -\frac{d}{2}e^{-\frac{2}{d}(\varphi(t_1)-Ct_1)}+c_1 t_1\leq0.
    $$
 Since $t_1$ is arbitrary, this means that   
    \be\label{bound:heat_comp:entr}
    \int_ M  \mu_0(t)\log \mu_0(t)  \,dx=\varphi(t)\le-\frac{d}{2}\log\left(\frac {2c_1}dt\right)+C(d, M )t
    \ee
    for every $t>0$.\\
    Now, for any given $ \beta>1$, we consider the reparametrization $\Tilde{\mu}_0(t,\cdot)\coloneqq\mu_0(t^\beta)$ for every $t>0$. Its velocity field is
    $$
    \Tilde{\nu}_0(t,\cdot)\coloneqq \beta t^{\beta-1}\nu_0(t^\beta,\cdot)
    $$
    so, for any fixed $0<\delta_0<\frac{T}{3}$, by \eqref{bound:heat_comp:vel}
    \begin{align*}
        \int_0^{\delta_0}\!\!\int_ M \abs{\Tilde{\nu}_0}^2\Tilde{\mu}_0\,dxdt&=\frac{1}{\beta}\int_0^{\delta_0^\beta}\!\!t^{1-\frac{1}{\beta}}\int_ M  \abs{\nu_0}^2\mu_0\,dxdt\\
        & \leq\frac{1}{\beta}\int_0^{\delta_0^\beta}\!\!\frac{1}{t^{\frac{1}{\beta}}}\,dt
    \end{align*}    
    which is finite for every $\beta>1$. With such a choice, if we merge this estimate  with \eqref{bound:heat_comp:entr} we obtain
    $$
        \int_0^{\delta_0}\!\!\int_ M \abs{\Tilde{\nu}_0}^2\Tilde{\mu}_0\,dxdt+\vep\int_0^{\delta_0}\!\int_ M  \Tilde{\mu}_0(t)(\log \Tilde{\mu}_0(t)+ V)\,dxdt\leq C_0
    $$
    where $C_0$ is a constant depending only on $ M, d, \lambda, T,  \|\vep V\|_\infty, \beta$ and $\delta_0$.\\
    In a similar way, for a fixed $T/3<\delta_1<T$, we find a smooth curve of probability densities $\Tilde{\mu}_1$, with velocity $\Tilde{\nu}_1$ such that
    $$
        \int_{\delta_1}^T\!\!\int_ M \abs{\Tilde{\nu}_1}^2\Tilde{\mu}_1\,dxdt+\vep\int_{\delta_1}^T\!\int_ M  \Tilde{\mu}_1(t)(\log \Tilde{\mu}_1(t)+ V)\,dxdt\le C_1
    $$
    where $C_1$ is a constant depending only on $ M, d, \lambda, T,  \|\vep V\|_\infty, \beta$ and $\delta_1$.\\
    Consider now the 2-Wasserstein geodesic $(\overline{\mu},\overline{\nu})$ between $\Tilde{\mu}_0(\delta_0,\cdot)$ and $\Tilde{\mu}_1(\delta_1,\cdot)$. We recall (see e.g. \cite{Daneri_2008})   that the entropy functional is $(-\lambda)$-convex along the 2-Wasserstein geodesic. Hence, by \eqref{bound:heat_comp:entr},
    $$
        \int_{\delta_0}^{\delta_1}\!\!\int_ M \abs{\overline{\nu}}^2\overline{\mu}\,dxdt+\vep\int_{\delta_0}^{\delta_1}\!\!\int_ M \overline{\mu}(\log\overline{\mu}+ V)\,dxdt\le C
    $$
    where the constant $C$ depends only on $ M, d, \lambda, T ,  \|\vep V\|_\infty, \delta_0,\delta_1$ and the Wasserstein distance $W_2\bigl(\Tilde{\mu}_0(\delta_0,\cdot),\Tilde{\mu}_1(\delta_1,\cdot)\bigr)$. However, the latter is uniformly estimated in terms of the manifold, 
    thanks to the compactness of $ M $. 
    Finally, gluing the paths from $m_0$ to $\tilde \mu_0(\delta_0,\cdot)$ and from  $\tilde \mu_1(\delta_1,\cdot)$ to $m_1$ with the 2-Wasserstein geodesic $\overline{\mu}$, we have built an admissible arc joining $m_0$ and $m_1$, and with a convenient choice of $\delta_0, \delta_1$ we estimate 
    $$
    \inf \mathcal{F}_\vep(m_0,m_1)\leq C_0+C +C_1
    $$
    where last constants only depend on $M, d, \lambda, T,  \|\vep V\|_\infty$. 
    
    It is a classical result, after Benamou-Brenier's trick \cite{BB} and the weak-lower semicontinuity of the entropy, that  the above estimate, with the existence of an admissible curve, yields the existence of a minimizer of $\cF_\vep$ by direct methods of Calculus of Variations. For a similar proof in euclidean context, see e.g. \cite[Proposition 2.9]{LaSa}.
    \end{proof}

\begin{remark}
\label{T-dependence}\rm
With a suitable choice of $\de_0,\de_1$ in the above proof, it is possible to give an estimate of the dependence of the constant $C$ in \rife{upper-min} from the time-horizon $T$. Alternatively, if we denote $ \min \mathcal{F}_\epsilon^1$ the minimum in unit time $T=1$, with a simple time-scaling one can estimate
$$
 \min \mathcal{F}_\epsilon\leq \max\left(\frac1T, T\right)  \min \mathcal{F}_\epsilon^1 \leq \max\left(\frac1T, T\right) C(M,d,\lambda)
$$
\end{remark}

\begin{remark}\label{dasoli} \rm We stress that in the above proof we can avoid the use of the  Wasserstein geodesic $(\overline{\mu},\overline{\nu})$ between $\Tilde{\mu}_0(\delta_0,\cdot)$ and $\Tilde{\mu}_1(\delta_1,\cdot)$ (and consequently, avoid the use of the $(-\lambda)$- displacement convexity of the geodesic). In fact, since $\Tilde{\mu}_0(\delta_0,\cdot)$ and $\Tilde{\mu}_1(\delta_1,\cdot)$ are smooth and positive, we can take the smooth optimal curve  of $\cF_\vep$ joining the two measures, whose existence will be proved in Section \ref{sec5}, Theorem \ref{smoothex}.
\end{remark}


 \section{The optimality system}

In this Section we discuss the structure  of the optimality system satisfied by minima of functional \rife{func-precise}.  
This is a first order PDE system which takes the following form
\begin{equation}
\label{mfg-log}
\left\{
\begin{aligned}
    &-\partial_t u+\frac{1}{2}\abs{\nabla u}^2=   \epsilon (\log(m)+V )\,, \qquad t\in (0,T)\\
    & \partial_t m-div_g(m\nabla u)=0\,, \qquad \qquad \qquad t\in (0,T).
\end{aligned}
\right.
\end{equation}
Let us first underline that, at least formally, \rife{mfg-log} is the optimality condition of \rife{func-precise}, in the sense that $(m,\nabla u)$ provides with the  couple $(m,v)$ minimizing \rife{func-precise}.  This is a consequence of the convexity of the functional, following the original idea of Benamou and Brenier  \cite{BB}.  Even if this is clear to expert readers, we provide a proof for completeness.

\begin{lemma}\label{itsmin} Let $(u,m)$ be a  smooth solution of system \rife{mfg-log}, in the sense that $u\in  C^1([0,T]\times  M ), m\in C^0([0,T]\times  M ) $ with $m(0)=m_0, m(T)=m_1$, and  $m>0$. Then $(m,\nabla u)$ is a minimum  point of \rife{func-precise}.
\end{lemma}

\begin{proof} Let $(\mu,v)$ be any couple which solves the continuity equation  in the sense of Definition \ref{def-conteq-sol}. We can assume that $\cF_\vep(\mu,v)<\infty$ (otherwise, the inequality  $\cF_\vep(m,\nabla u)\leq \cF_\vep(\mu,v)$ is obvious), and in particular $\mu \in L^1(Q_T)$.  Let  us define the following convex and lower semicontinuous function in $\R^d\times \R$:
\be\label{BB}
 \Psi(p,m)= \left\{
    \begin{array}[c]{ll}
      \frac{|p|^2}{2m}  \quad & \hbox{if}\quad m>0,\\
      0\quad & \hbox{if}\quad m=0 \hbox{ and } p=0,\\
      +\infty \quad &\hbox{otherwise}.
    \end{array}
\right.
\ee
We set $w=\mu v, \hat w= m\nabla u$. Since $m,u$ are smooth solutions of \rife{mfg-log} (with $m>0$),  then $\partial \Psi(\hat w,m)$ is well defined.  By convexity of $\Psi$, we have
\begin{align*}
\Psi(\hat w, m)- \Psi (w,\mu) & \leq \partial_p \Psi(\hat w, m) \cdot (\hat w-w) + \partial_m \Psi(\hat w,m)(m-\mu)
\\ & = \frac{\hat w}m\cdot (\hat w-w)-  \frac{|\hat w|^2}{2m^2}(m-\mu)
\end{align*}
Since  $\mu\in L^1$ and $\mu|v|^2 \in L^1$, we have $w\in L^1$ and the above inequality is integrable on $ M $. From the very definition of $w,\hat w$ we deduce
\be\label{Psi}
\iinto  \frac{1}{2} m|\nabla u|^2 \leq \frac12 \iinto \mu|v|^2 + \frac12 \iinto m |\nabla u|^2 - \iinto w\cdot \nabla u + \frac12 \iinto \mu |\nabla u|^2
\ee
Since $u\in C^1$,   the continuity equation gives 
\begin{align*}
\iinto w\cdot \nabla u & =  \iinto \partial_tu\, \mu- \into u(T)m_1+ \into u(0)m_0  
\\ & = -\vep \iinto (\log m+ V)\mu + \frac12 |\nabla u|^2 \mu -
\into u(T)m_1+ \into u(0)m_0\,.
\end{align*}
Hence from \rife{Psi} we get
\begin{align*}
\iinto  \frac{1}{2} m|\nabla u|^2 & \leq \frac12 \iinto \mu|v|^2 + \frac12 \iinto m |\nabla u|^2 +
\into u(T)m_1- \into u(0)m_0  \\ & + \vep \iinto (\log m+ V)\mu
\\ & = \frac12 \iinto \mu|v|^2 - \vep \iinto m(\log(m)+ V)+ \vep \iinto (\log m+ V)\mu
\end{align*}
By convexity we obviously have $m\log(m)-\mu\log(\mu)\leq \log(m) (m-\mu)+(m-\mu)$, where last term disappears after integration. Then we conclude that
$$
\mathcal{F}_\epsilon(m,\nabla u) \leq \mathcal{F}_\epsilon(\mu,v)\,.  
$$
\end{proof}

Since \rife{mfg-log} is a Hamiltonian system, there is  some invariant of motion. The proof is straightforward. 

\begin{lemma}\label{HS}  Let $(u,m)$ be a smooth solution of system \rife{mfg-log}.   Then we have that the quantity 
$$
 E(m_0,m_1):=\frac12 \into m|\nabla u|^2 - \vep \into m(\log(m)+V)
$$
is constant in time and it holds
\begin{equation}    
    \label{bound:E}
  E(m_0,m_1)= \frac1T  {\mathcal B}_\vep(m_0,m_1)- \frac{2\vep}T \iinto m(\log(m)+V),
\end{equation}
      where ${\mathcal B}_\vep(m_0,m_1)=  \min \mathcal{F}_\epsilon$. 
\end{lemma}

{We observe that the quantity $E$ only depends on the marginals $m_0, m_1$, and is easily estimated. In particular, by Jensen's inequality,  we have $\into m(\log m +V) dx\geq -\log (\nu(M))$, for the measure  $\nu:= e^{-V}dx$. Recalling  Proposition \ref{prop:exist_comp} (and Remark \ref{T-dependence}), we deduce that
\be\label{stimaE}
E(m_0,m_1)\leq \frac1{T^2 \wedge 1}\, C(M,d,\lambda) + 2\vep \log (\into e^{-V}dx) \leq K
\ee
for some constant $K$ which is uniform for all $m_0, m_1\in \cP(M), V\in L^\infty(M)$ and any  $\vep \leq 1$.} 
 \subsection{Displacement convexity estimates}

In this section we study the convexity of some energy functional  along the optimal curves of   \rife{func-precise}.  This is obtained in the Eulerian approach by exploiting dissipativity properties of the solutions of system \rife{mfg-log}. We consider smooth solutions, which justifies the computations below; as we will see later, this is no loss of generality, since all solutions will be obtained as limit of classical ones. The following is an extension to the Riemannian setting of the results proved in \cite{GoSe}, (or in \cite{Po} with Neumann conditions); the only new ingredient is provided by the B\"ochner formula \rife{boch}.

\begin{proposition}\label{prop:conv}
Let $u\in C^2(\overline{Q}_T)$ and $m\in C^1(\overline{Q}_T)$ be classical solutions to the system \rife{mfg-log}, where $V\in W^{2,\infty}( M )$.

Let $U:(0,+\infty)\rightarrow\R$ be a $C^1$ function such that
\begin{equation*}
    P(r)\coloneqq U'(r)r-U(r)\ge0
\end{equation*}
Then 
\begin{equation}\label{disco}
\begin{split}
    \frac{d^2}{dt^2}\int_ M  U(m)\,dx\ge&\int_ M \left[P'(m)m-(1-\frac{1}{d})P(m)\right](\Delta_gu)^2\,dx\\
    &+\int_ M  P(m)Ricc_g(\nabla u,\nabla u)\,dx+\\
    &+\int_ M  \epsilon\frac{P'(m)}{m}\abs{\nabla m}^2\,dx  + \vep \int_ M  P'(m) \nabla m \scalg \nabla V\,dx 
\end{split}\end{equation}
\end{proposition}
\begin{proof}
    We begin by calculating the first derivate of the function $t\rightarrow \int_ M  U(m)\,dx$,
\begin{equation*}
\begin{aligned}
    \frac{d}{dt}\int_ M  U(m)\,dx&=\int_ M  U'(m)\partial_t m\,dx\\
                                    &=\int_ M  U'(m)(m\Delta_g u+\nabla m\scalg \nabla u)\,dx\\
                                    &=\int_ M  P(m)\Delta_g u\,dx
\end{aligned}\end{equation*}
recalling that $P(r)= U'(r)r - U(r)$. So, the second derivative takes the form 
\begin{align*}
    \frac{d^2}{dt^2}\int_ M  U(m)\,dx=&\int_ M  P'(m)\partial_t m\Delta_g u+P(m)\Delta_g (\partial_tu) \,dx \\                           
                                    =&\int_ M  P'(m)(m\Delta_g u+\nabla m\scalg \nabla u)\Delta_g u\,dx+\\
                                    &+\int_ M  P(m)\Delta_g(\frac{1}{2}\abs{\nabla u}^2-\epsilon (\log(m)+V))\,dx
                                    \\  = &  \int_ M [P'(m)m-P(m)](\Delta_g u)^2\,dx-\int_ M  P(m)\nabla (\Delta_g u)\scalg \nabla u\,dx
                                    \\ & +\int_ M  P(m)\Delta_g(\frac{1}{2}\abs{\nabla u}^2-\epsilon (\log(m)+V))\,dx
                                    \,.
                                    \end{align*}
Now we use B\"{o}chner's formula \eqref{boch} to calculate $\Delta_g(\frac{1}{2}\abs{\nabla u}^2)$, obtaining
\begin{align*}
  \frac{d^2}{dt^2}\int_ M  U(m)\,dx= 
                                   &\int_ M [P'(m)m-P(m)](\Delta_g u)^2\,dx\\
                                    &+\int_ M  P(m)[tr( (\nabla^2u)^2)+Ricc_g(\nabla u, \nabla u)]\,dx\\
                                    &+\int_ M  \epsilon\frac{P'(m)}{m}\abs{\nabla m}^2dx+ \vep \int_ M  P'(m) \nabla m \scalg \nabla V\,dx.\\
\end{align*}
Since, by the symmetry of $\nabla^2u$, it holds $tr( (\nabla^2u)^2)\ge\frac{1}{d}(tr(\nabla^2u))^2=\frac{1}{d}(\Delta_gu)^2$, we get \rife{disco}.
\end{proof}

    
In particular, the inequality \rife{disco} implies the semi-convexity of the $\log$-entropy along the optimal curve $m(t)$, and   the strict convexity of the relative entropy whenever $Ricc_g + D^2 V\geq 0$.  

\begin{corollary}\label{corollary:semi_conv:log}
    Under  the assumptions of Proposition \ref{prop:conv}, let $\cH(m(t); \nu)$ be the relative entropy defined in \rife{rel-entr}, for $\nu=e^{-V}dx$. Then we have
     \be\label{BK+displa}
\begin{split}
\frac{d^2}{dt^2} \cH(m(t); \nu) & \geq \into m (Ricc_g(\nabla u,\nabla u)+ D^2 V(\nabla u,\nabla u)) \, dx  \\ & \qquad\qquad + \vep \into |\nabla (\log m + V)|^2\, m\, dx\,.
\end{split}
     \ee
     Moreover,  let $\lambda\geq 0$ satisfy \rife{ricci},  and define 
     $$
     \vfi(t)= \int_ M  m(t)\log m(t)\, dx\,.
     $$ 
Then we have:
     \begin{itemize}
     \item[(i)] there exists a constant 
     $\Lambda_\vep$, depending on $M,d,\lambda,  \|V\|_{W^{1,\infty}}, T, \vep$, such that $\vfi$ is $\Lambda_\vep -$ semiconvex in $(0,T)$, hence
     \begin{align}
     \label{bound:semi_conv:log}
         \varphi(t)\le\frac{T-t}{T}\varphi(0)+\frac{t}{T}\varphi(T)+\Lambda_\vep \frac{t(T-t)}{2T^2}
     \end{align}
  Moreover, the constant $\Lambda_\vep$ is bounded independently of $\vep$ (for $\vep\leq 1$),  and we have $\Lambda_\vep \,  \mathop{\to}\limits^{\vep \to 0} \, \frac{\lambda}TW_2(m_0,m_1)^2$.
%
\vskip0.4em
\item[(ii)] there exists a constant $L=L(M,d,\lambda,  \|V\|_{W^{1,\infty}}, T)$ such that
\be\label{loc_bound_entropy}
\vfi(t) \leq  d\, | \log (t(T-t)) | + \frac d2 \, |\log \vep |+ L  \qquad \forall t \in (0,T)\,. 
\ee
\end{itemize}
\end{corollary}

\begin{proof}
    We use Proposition \ref{prop:conv} with $U(r)=r\log r -r$ (so that $P(r)=r$) and we get
\be\label{Plog}\begin{split}
   \frac{d^2}{dt^2}\int_ M  m\log m= &   \frac{d^2}{dt^2}\int_ M  (m\log m-m)\,dx\ge \int_ M  mRicc_g(\nabla u,\nabla u)\,dx+\\
    &\qquad +\int_ M  \epsilon\frac{1}{m}\abs{\nabla m}^2\,dx+ \vep \int_ M  \nabla m\scalg \nabla V\,dx\,.
    \end{split}
    \ee
Similarly, we compute
\begin{align*}
 \frac{d^2}{dt^2}\int_ M  m\,V = &  \frac{d}{dt }\int_ M  \partial_t m\,V = -  \frac{d}{dt } \into m\, \nabla V\scalg \nabla u \\
 & = - \into (\nabla V\scalg \nabla u) div_g(m \nabla u) - \into m \, \nabla V\scalg \nabla ( \frac12 |\nabla u|^2- \vep (log(m) + V))
 \\
 & = \into m\, D^2 V (\nabla u, \nabla u) + \vep \into m \, \nabla V\scalg \nabla  (log(m) + V)\,.
\end{align*}    
Adding this equality to \rife{Plog}, we obtain \rife{BK+displa}.    
    
Now, if we come back to \rife{Plog} and use the lower bound on $Ric_g$, we get        
$$
 \frac{d^2}{dt^2}\int_ M  m\log m \geq -\lambda\int_ M  m\abs{\nabla u}^2\,dx+\frac\vep 2 \int_ M   \frac{1}{m}\abs{\nabla m}^2\,dx- \frac\vep 2\int_ M    m|\nabla V|^2\,dx\,.
$$
By  definition  of the quantity $E(m_0,m_1)$ we obtain
    \be\label{serve}
    \begin{split}   \frac{d^2}{dt^2}\int_ M  m\log m \ge
    &- 2\lambda E(m_0,m_1)-2\lambda\epsilon  \int_ M  m (\log m +V)  dx\\
    &+\frac\vep 2 \int_ M  \frac{1}{m}\abs{\nabla m}^2\,dx- \frac\vep2 \int_ M    m|\nabla V|^2\,dx
\\
    \ge&-2\lambda E(m_0,m_1)-2\lambda\epsilon\int_ M  m\log m  dx\\
    &+\frac\vep 2 \int_ M  \frac{1}{m}\abs{\nabla m}^2\,dx- \vep \, c(\lambda, \|V\|_{W^{1,\infty}})  .
\end{split}
\ee
We estimate the Fisher information of $m$ as in Proposition \ref{prop:exist_comp}, see \rife{fish}:
$$
\int_ M  \frac{1}{m}\abs{\nabla m}^2\,dx\geq c_1 \exp\left( \frac2d\int_ M  m\log m \right) - c_2 Vol(M)\,.
$$
Therefore, if $\varphi(t)\coloneqq\int_ M  m\log m \,dx$,  we deduce
\begin{equation}\label{2nd-order-semiconv-vep}
    \varphi''\ge-2\lambda E(m_0,m_1)+\epsilon( -2\lambda\varphi+c_3e^{\frac2 d\varphi})- \vep \, c(\lambda, \|V\|_{W^{1,\infty}},  M ) \,,
\ee
for some constant $c_3$.  We note that the function $r\to -2\lambda r+ c_3 e^{\frac2d r}$ has a finite minimum on $[0, +\infty)$, and that $E(m_0,m_1)$ is bounded above by some constant $K$ only depending on $M,d,\lambda, T, \|V\|_\infty$, see \rife{stimaE}. Hence we have
$$
 \varphi''\ge-2\lambda K-\epsilon C(\lambda, \|V\|_{W^{1,\infty}},  M )
 $$
 which gives the semiconvexity of the   entropy  along the optimal curves, with a semi-convexity constant $\Lambda_\vep $ which is bounded  uniformly for $\vep\leq 1$. In a more precise form,  on account  of \rife{bound:E} we can estimate
 $$
 \Lambda_\vep = \vep C+ 2\lambda E(m_0, m_1) \simeq \vep C (1+ |\log \vep |) + 2\frac {\lambda }T \min(\cF_\vep)\,.
 $$
 As we will prove in Section \ref{sec6}, it holds that $\min(\cF_\vep) \to \frac12W_2(m_0,m_1)^2$; hence we deduce 
 $$
 \Lambda_\vep  \, \mathop{\to}^{\vep \to 0} \, \,  \frac {\lambda }TW_2(m_0,m_1)^2\,.
 $$
%
%

Now we also obtain a local bound for the entropy, independently from the initial and terminal marginals. Indeed, we deduce from \rife{2nd-order-semiconv-vep} that
$$
 \varphi''\ge   \vep\, c_4  e^{\frac2 d\varphi}- c_5 \qquad t\in (0,T),
$$
for some constant $c_4, c_5$ depending on $M,d,\lambda, \|V\|_{W^{1,\infty}}, T, \vep$ (and uniform for $\vep\leq 1$). With a suitable choice of $L$ (depending on $ c_4,c_5, T$), we have that the function 
$$
\psi(t):= -d \log ( \sqrt\vep\, t(T-t)) + L
$$
is a supersolution of the same equation, i.e.  $\psi''\le   \vep\, c_4  e^{\frac2 d\psi}- c_5$ for $t\in (0,T)$. Since $\psi$ blows-up at $t=0, t=T$, we conclude by comparison that $\vfi\leq \psi$, which  gives \rife{loc_bound_entropy}.
\end{proof}

\subsection{The optimality system as an elliptic equation}\label{1st-to-2nd-order}

System \rife{mfg-log} can be recasted as a single elliptic equation,  in time-space variables,  for $u$. This comes by noting that  $m e^V= \exp(\frac1\vep(\frac{|\nabla u|^2}2-\partial_t u))$, which can be inserted in the continuity equation, giving rise to a quasilinear elliptic equation in divergence form for $u$. 

This is in fact a special case of a general approach suggested by P-L. Lions in his  lectures at Coll\`ege de France \cite{L-college}, in order to handle mean-field game  systems of first order, such as
 \begin{equation}
\label{elliptic:1st_order}
\left\{
\begin{aligned}
    &-\partial_t u+\frac{1}{2}\abs{\nabla u}^2= f(m)  +V \\
    &\partial_t m-m\Delta_gu-\nabla u\scalg \nabla m=0
\end{aligned}
\right.
\end{equation}
whenever $f$ is an increasing function.  For the reader's convenience, we derive here this equation in the Riemannian context. To this purpose, we first compute the  
covariant gradient of both terms in the Hamilton-Jacobi equation:
\begin{equation*}
\begin{split}
    f'(m) \nabla m
    &= \nabla \left(-\partial_t u+\frac{\abs{\nabla u}^2}{2}-V\right)\,.
\end{split}
\end{equation*}
Taking the scalar product with $\nabla u$ we get
$$
 f'(m) \nabla m \scalg \nabla u= \nabla \left(-\partial_t u+\frac{\abs{\nabla u}^2}{2}-V\right)\scalg \nabla u\,.
$$
Similarly, by taking the time derivative from the Hamilton-Jacobi equation, we have:
\begin{equation}\label{measure:deriv_time}
\begin{split}
    f'(m) \partial_t m
    &= \frac{\partial}{\partial t}\left(-\partial_t u+\frac{\abs{\nabla u}^2}{2}-V\right)\\
    &=  -\partial_{tt} u +\nabla (\partial_t u) \scalg \nabla u\,.\\    
\end{split}
\end{equation}
From the above equalities, merged with the continuity equation for $m$, we obtain
 \be\label{generica}\begin{split}
   -\partial_{tt} u +\nabla (\partial_t u)\scalg \nabla u -  \nabla \left(-\partial_t u+\frac{\abs{\nabla u}^2}{2}-V\right)\scalg \nabla u  &   =   f'(m) \left( \partial_t m -\nabla m\scalg \nabla u\right) \\
& =f'(m) \,m \Delta_g u 
\end{split}
\ee
which becomes a second order equation in the only unknown  $u$. Indeed, by setting  $\phi\coloneqq (f)^{-1}$, the first equation of the system reads as
\begin{equation*}
    m=\phi\left(-\partial_t u+\frac{\abs{\nabla u}^2}{2}- V\right)\,,
\end{equation*}
and so \rife{generica} becomes  a single equation for $u$.
%
%
In the particular case of $f(m)=\vep  \log m$, we have $f'(m)m=\vep$, and \rife{generica} simplifies further; if we also scale the potential $V$ into  $\vep\, V$ (this is not necessary of course, but it is more consistent with the model of relative entropy), we get to the following form:
\begin{equation}
\label{elliptic:long}
    -\partial_{tt} u  +2\nabla u\scalg \nabla (\partial_t u)-(\nabla^2u)(\nabla u, \nabla u)-\vep\Delta_g u+ \vep \nabla u\scalg \nabla V=0.
\end{equation}
We can shorten such equation by writing it as a quasilinear equation in the space-time variable
\begin{equation}
\label{elliptic:short}
    -tr\left(\mathcal{A}(x,\nabla u)\circ\overline{\nabla}^2u\right)+ \vep \nabla u\scalg \nabla V(x)=0
\end{equation}
where 
$\mathcal{A}(x,\eta)$ is the endomorphism of $\R\times T_x\! M $ defined by 
\begin{equation*}
    \begin{array}{ccl}
    \R\times T_x\! M &\longrightarrow&\R\times T_x\! M\\
    \relax[\overline{w},w]&\longrightarrow&[-(\eta\scalg w-\overline{w}),(\eta\scalg w-\overline{w})\eta+\vep w]
\end{array}
\end{equation*}   
and, for every $C^2$ function $f$ on  $[0,T]\times  M $, 
the endomorphism $\overline{\nabla}^2f$ is given by 
\begin{equation*}
    \begin{array}{ccl}
    \R\times T_x\! M &\longrightarrow&\R\times T_x\! M\\
\relax[\overline{w},w] & \longrightarrow & [\overline{w}\partial_{tt}f+\nabla(\partial_t f)\scalg w, \overline{w}\nabla(\partial_t f)+D_w\nabla f]
\end{array}
\end{equation*}

 Note that $\mathcal{A}(x,\eta)$ is independent of $t\in[0,T]$ and it is symmetric for every choice $(x,\eta)\in  M \times T_x M $. The symbol $\mathcal{A}(x,\eta)$ will denote also the bilinear form induced by such endomorphism through the product metric $\underline{g}$ of the manifold $\R\times M$.  Namely, 
 \begin{align*}
     \mathcal{A}(x,\eta)([\overline{w},w],[\overline{v},v])&\coloneqq\left(\mathcal{A}(x,\eta)[\overline{w},w]\right)\cdot_{\underline{g}}[\overline{v},v]\\
     &=(\eta\scalg w-\overline{w})(\eta\scalg v-\overline{v})+\vep w\scalg v\,.
 \end{align*}
Finally we note that such bilinear form is elliptic (though not uniformly), in fact for every $[\overline{w},w]\in\R\times T_x\! M $ we have
\begin{equation*}
    \mathcal{A}(x,\eta)([\overline{w},w],[\overline{w},w])=(\eta\scalg w-\overline{w})^2+\vep\abs{w}^2>0 \qquad \forall [\overline{w},w]\neq [0,0].
\end{equation*}

\section{Gradient bounds for smooth solutions}

In this section we obtain estimates for smooth solutions of the system \rife{opsys}, by  exploiting the elliptic character of  the quasilinear equation   \rife{elliptic:short}. We mostly follow the ideas developed in \cite{L-college}, \cite{Mu}, \cite{Po}, although specifying to the case of  the entropy nonlinearity allows us to simplify some argument and to give estimates in a more precise form.

As a first step, we derive from  \eqref{elliptic:long}  the differential equations solved by some auxiliary functions of $u$ and its derivatives.

\begin{lemma}\label{lemma:calc}
    Let $u\in C^3([0,T]\times  M )$ be a solution of 
    \begin{equation}\label{eq:lemma_calc}
        -tr\left(\mathcal{A}(x,\nabla u)\circ\overline{\nabla}^2u\right)+\rho u+\vep \nabla u\scalg \nabla V(x)=0
    \end{equation}
    Then
    \begin{enumerate}[label=\roman*)]
        \item for every $K\in\R$, the function $h\coloneqq(u+K)^2$ satisfies
        \begin{align*}
            -tr\left(\mathcal{A}(x,\nabla u)\circ\overline{\nabla}^2h\right) + \vep \nabla u\scalg \nabla V & + 2 \rho u\,(u+K) \\
            & = -2\mathcal{A}(x,\nabla u)\left([\partial_t u ,\nabla u],[\partial_t u,\nabla u]\right).
        \end{align*}
         \item Set $\omega\coloneqq \partial_t u$, then it satisfies
        \begin{align*}
            -tr\left(\mathcal{A}(x,\nabla u)\circ\overline{\nabla}^2\omega\right)+ \vep \nabla \omega\scalg \nabla V+ \rho\omega=&-2\abs{\nabla \omega}^2+2(\nabla^2u)(\nabla u, \nabla \omega).
        \end{align*}
         \item Set $\varphi\coloneqq \frac{1}{2}\abs{\nabla u}^2$, then it satisfies
        \begin{align*}
            -tr\left(\mathcal{A}(x,\nabla u)\circ\overline{\nabla}^2\varphi\right) & + \vep \nabla \varphi\scalg \nabla V + 2\rho\varphi  = \abs{\nabla \varphi}^2 -| \nabla (\partial_t u)|^2  \\
            & - \vep \left( \abs{\nabla^2u}^2+ (\nabla^2V)(\nabla u, \nabla u)+Ricc_g(\nabla u,\nabla u)\right)  \\
        \end{align*}
    \end{enumerate}
\end{lemma}
\begin{proof}
    Equations i) and ii) are straightful computations based on equation \eqref{eq:lemma_calc}. For i) we use the chain rule. For ii) we derive in time equation \eqref{eq:lemma_calc} and we use that $\mathcal{A}(x,\nabla u)$ is independent from time.

        iii) We now study the differential equation solved by $\varphi\coloneqq\tfrac{1}{2}\abs{\nabla u}^2$. We first observe that, if we develop the trace in \rife{eq:lemma_calc}, as we did in \rife{elliptic:long},  we can rewrite the equation as 
    \begin{equation}\label{lemma:3:first_phi}
   \partial_{tt} u+\vep\Delta_g u=2\partial_t \varphi -\nabla u\scalg \nabla \varphi+\vep \nabla u\scalg \nabla V+\rho u
    \end{equation}
    Now, using the Bochner's formula \rife{boch},  which means
    \begin{equation}\label{bochner}
        \Delta_g\varphi=\tfrac{1}{2}\Delta_g\abs{\nabla u}^2=\abs{\nabla^2u}^2+\nabla (\Delta_gu)\scalg \nabla u+Ricc_g(\nabla u, \nabla u)
    \end{equation}
    we get
    \begin{align*}
        \partial_{tt}  \varphi   +\vep\Delta_g\varphi & = 
            \nabla (\partial_t u)\scalg \nabla (\partial_t u) +\nabla u\scalg \nabla (\partial_{tt} u) 
            +\vep\left(\nabla (\Delta_gu)\scalg \nabla u +\abs{\nabla^2u}^2+Ricc_g(\nabla u, \nabla u)\right)
            \\
        &= \nabla \bigl(\partial_{tt} u +\vep\Delta_gu\bigr)\scalg \nabla u+ |\nabla (\partial_t u)|^2  
        +\vep\left[\abs{\nabla^2u}^2+Ricc_g(\nabla u, \nabla u)\right]\,.
        \end{align*}
Then, using \eqref{lemma:3:first_phi}, we have
         \begin{align*}
        \partial_{tt}\vfi   +\vep\Delta_g\varphi  
        &= 2\nabla (\partial_t \varphi) \scalg \nabla u-(\nabla^2u)(\nabla u, \nabla \varphi)-(\nabla^2\varphi)(\nabla u, \nabla u)+\\
        &\quad +\vep(\nabla^2u)(\nabla u, \nabla V)+ \vep (\nabla^2V)(\nabla u, \nabla u)+2\rho\varphi+\\
        &\quad +|\nabla (\partial_t u)|^2  + \vep\left[\abs{\nabla^2u}^2+Ricc_g(\nabla u, \nabla u)\right]\,.
\end{align*} 
We note that for every vector field $v$ on $M$ we have $\nabla \varphi\scalg v=(\nabla^2u)(\nabla u, v)$, so the above equality becomes
 \begin{align*}
        \partial_{tt}\vfi   +\vep\Delta_g\varphi & = 
         2\nabla (\partial_t \vfi) \scalg \nabla u-\nabla \varphi\scalg \nabla \varphi-(\nabla^2\varphi)(\nabla u, \nabla u)+\vep \nabla \varphi\scalg \nabla V+\\
        &+\vep (\nabla^2V)(\nabla u, \nabla u)+2\rho\varphi+ |\nabla (\partial_t u)|^2  +\\
        &+\vep\left[\abs{\nabla^2u}^2+Ricc_g(\nabla u, \nabla u)\right].
        \end{align*}
Finally, if we look at the whole chain of equalities we get 
\begin{align*}
           &  -tr\left(\mathcal{A}(x,\nabla u)\circ\overline{\nabla}^2\varphi\right) + \vep \nabla \varphi\scalg \nabla V+ 2\rho\varphi= \abs{\nabla \varphi}^2
           -| \nabla (\partial_t u)|^2 
            \\ & \qquad\qquad \qquad    -\vep \left( \abs{\nabla^2u}^2+(\nabla^2V)(\nabla u, \nabla u)+ Ricc_g(\nabla u,\nabla u)\right)\,.
             \end{align*}
\end{proof}
\subsection{Global Lipschitz bound on $u$}
In this section we will prove a $W^{1,\infty}$ bound for the solution $u\in C^3(Q_T )$ of the system 
\begin{equation}
\label{elliptic:2nd_order:par}
\begin{cases}
    -tr\left(\mathcal{A}(x,\nabla u)\circ\overline{\nabla}^2u\right)+\rho u+ \vep \nabla u\scalg \nabla V(x)=0 &\text{in $Q_T $}\\
    -\partial_t u+\frac{1}{2}\abs{\nabla u}^2=\delta u+\epsilon (\log(m_1)+V(x) )&\text{in $t=T,x\in  M $}\\
    -\partial_t u+\frac{1}{2}\abs{\nabla u}^2+\delta u=\epsilon( \log(m_0)+V(x) )&\text{in $t=0,x\in  M $.}\\
\end{cases}
\end{equation}
We recall that $Q_T= (0,T)\times M$, and hereafter we denote 
\begin{gather*}
    \Sigma_0\coloneqq\{0\}\times  M \quad\Sigma_T\coloneqq\{T\}\times  M .
\end{gather*}
We notice that \rife{elliptic:2nd_order:par} is a perturbation of  \rife{elliptic:short}, containing an additional term $\rho u$ in the interior (this will simplify a preliminary step, detailed in Appendix) and an additional term $\de u$ on the boundary. The latter  is used to control the function $u$ in a first step. Indeed, using  the maximum principle, we have
\be\label{deltau}
 \delta\norma{u}_\infty\le\left(\norma{\epsilon(\log(m_0)+V)}_\infty+\norma{\epsilon(\log(m_1)+V)}_\infty\right).
\ee
%
Analogously, using Lemma \ref{lemma:calc} and the maximum principle, we bound the time derivative.
\begin{lemma}
\label{lemma:u_t} Let $u$ be a solution of \rife{elliptic:2nd_order:par}. 
    It holds that
    \begin{equation}
        \partial_t u(t,x)\le \sup_{\Sigma_0\cup\Sigma_T}(\partial_t u)_+\quad \text{and}\quad\abs{\partial_t u(t,x)}\le \sup_{\Sigma_0\cup\Sigma_T}\abs{\partial_t u(t,x)}\quad\forall(t,x)\in Q_T
    \end{equation}
\end{lemma}
Finally we give a bound for the space derivative in terms of the $\sup$-norm of $u$.
\begin{theorem}
\label{thm:Lip_u_par} Let $u$ be a solution of \rife{elliptic:2nd_order:par}. 
    There exists a constant $C$, independent from $\rho$ and $\delta$, such that
    \begin{equation}\label{stimaDu}
        \norma{ {\nabla}u }_\infty\le C(1+\norma{u}_\infty)\qquad;\qquad   \norma{ {\partial_t}u }_\infty\le C(1+\norma{u}_\infty^2).
    \end{equation}
    Such constant $C$ depends on $ \norma{\epsilon\log(m_0)}_{W^{1,\infty}},\norma{\epsilon\log(m_1)}_{W^{1,\infty}}$ and on $\norma{\vep V}_{W^{2,\infty}}$.
\end{theorem}
\begin{proof}
We follow P.-L. Lions' method, as developed in \cite{Mu}, \cite{Po}. First of all, we replace $u$ with the auxiliary function
     \begin{equation*}
         v\coloneqq u+K-C_0\left(\tfrac{T-t}{T}\right)\,,\quad\text{where }K=2\norma{u}_\infty+1,\quad C_0=2K\,.
\end{equation*}
Notice that $\nabla v= \nabla u$, so $v$ solves the same elliptic equation of $u$, up to an additional term due to the time translation. Moreover, we have  $\norma{v}_\infty\le  C (1+\norma{u}_\infty)$. 
%
%
We set
\begin{equation*}
    z\coloneqq \tfrac{1}{2}\abs{\nabla v}^2+\tfrac{\gamma}{2}v^2
\end{equation*}
where
\begin{equation*}
    \gamma\coloneqq\frac{\sigma}{(1+\norma{u}_\infty)^2}
\end{equation*}
for some small constant $\sigma$ to be chosen later. The goal is to obtain an upper bound on $z$ by means of the maximum principle.
If the maximum occurs at the boundary, this means that either $t=0$ or $t=T$; here one uses that, by construction,  $v(T)\geq 1$, $v(0)\leq -1$, then  reasoning exactly as in \cite{Po} (Thm 3.4, Step 1) one obtains that  
\begin{equation*}
    \norma{\nabla v}_\infty\le C (1+\|u\|_\infty)  
    \end{equation*}
for some $C= C(\norma{ (\vep\log(m_1)}_{W^{1,\infty}}, \norma{(\vep\log(m_0)}_{W^{1,\infty}}, \|\vep \nabla V\|_\infty)$, in case of a maximum attained at the extremal times $t=0,T$.

Let us focus on a possibly interior maximum point. 
%
From  Lemma \ref{lemma:calc}, part (i) (applied to $v$) and part (iii), we have
\be\label{mezze}
\begin{split}
          &   -tr\left(\mathcal{A}(x,\nabla u)\circ\overline{\nabla}^2z\right) + \vep \nabla z\scalg \nabla V+ 2\rho\, z=  
            -\gamma\mathcal{A}(x,\nabla u)([\partial_t v,\nabla v],[\partial_t v,\nabla v])]+\\
            & \quad 
+ \abs{\nabla \varphi}^2- |\nabla (\partial_t v)|^2 -  \vep \left(\abs{\nabla^2u}^2+Ricc_g(\nabla u,\nabla u)+  (\nabla^2V)(\nabla u, \nabla u)\right) \\
            & \quad  +\gamma\rho(K-C_0 \tfrac{T-t}{T})v
\end{split}
\ee  
where $\vfi= \frac12 |\nabla u|^2$.  By definition  of the  matrix $\mathcal{A}(x,\nabla u)$, we get
$$
    \mathcal{A}(x,\nabla u)([\partial_t v,\nabla v],[\partial_t v,\nabla v])= \abs{-\partial_t v +\abs{\nabla v}^2}^2+\vep\abs{\nabla v}^2\\
$$
while the definition of $\vfi$ implies
$$
|\nabla \vfi|^2= \abs{\nabla z}^2-2\gamma v \nabla v\scalg \nabla z+\gamma^2v^2\abs{\nabla v}^2.
$$
Inserting the above equalities in \rife{mezze} we get
%
%
\begin{align*}
          &   -tr\left(\mathcal{A}(x,\nabla u)\circ\overline{\nabla}^2z\right)+\vep \nabla z\scalg \nabla V+2\rho z+\gamma\abs{-\partial_tv+\abs{\nabla v}^2}^2+\gamma\vep\abs{\nabla v}^2 \\
           & \quad  =   \abs{\nabla z}^2-2\gamma v \nabla v\scalg \nabla z+ \gamma^2v^2\abs{\nabla v}^2    - \abs{\nabla (\partial_t v)}^2
           \\ & \qquad - \vep 
             \left(\abs{\nabla^2u}^2+Ricc_g(\nabla u,\nabla u) +   (\nabla^2V)(\nabla u, \nabla u)\right)  +\gamma\rho(K-C_0 \tfrac{T-t}{T})v
\end{align*}
and since $Ric_g$ is bounded below,  and $\gamma v^2\leq C\sigma$ by  the initial choice, we can estimate the right-hand side obtaining
\begin{equation}
\label{thm:Lip_est:ineq2}    
 \begin{split}
&   -tr\left(\mathcal{A}(x,\nabla u)\circ\overline{\nabla}^2z\right)+\nabla z\scalg \nabla V+2\rho z+\gamma\abs{-\partial_t v +\abs{\nabla v}^2}^2+\gamma\vep\abs{\nabla v}^2 \\
           & \quad  \leq   \abs{\nabla z}^2-2\gamma v \nabla v\scalg \nabla z +  C ( 1+ |\nabla v|^2) 
\end{split}
\ee
for some constant $C$ depending on $(Ric_g+ \nabla^2 V)_- $ and on $\sigma$.
Let us focus on the quantity $-\partial_t v +|\nabla v|^2$: at an interior maximum point of $z$, we have
\begin{align*}
-\partial_t v+ |\nabla v|^2 & = (\max_{\overline{Q}}z-\partial_t u)-\frac{C_0}{T}-\frac{\gamma}{2}v^2+ \frac12  |\nabla v|^2
\end{align*}
However, thanks to Lemma \ref{lemma:u_t},  and to the boundary conditions at $t=0,T$, we have
\begin{align*}
    \partial_t u & \le\sup_{\Sigma_0\cup\Sigma_t} \partial_t u 
 \\ & \leq  \max_{\overline{Q}}z + \|\de u\|_\infty + \frac\gamma2 \|v\|_\infty^2+ \|\vep V\|_\infty + \vep \max \left( \|\log(m_0)\|_\infty, \|\log(m_1)\|_\infty\right)  
    \end{align*}
which implies, using $\gamma v^2\le C\sigma$ and \rife{deltau},  
    $$
    \partial_t u - \max_{\overline{Q}}z \leq \overline{K}
$$
for a certain $\overline{K}>0$ depending only on $\norma{\vep V}_\infty, \norma{\epsilon\log(m_1)}_\infty$ and $\norma{\epsilon\log(m_0)}_\infty$.  Therefore, we conclude that
\begin{align*}
-\partial_t v+ |\nabla v|^2 & \geq  - \overline{K} -\frac{C_0}{T}-\frac{\gamma}{2}v^2+ \frac12  |\nabla v|^2
\\ & \geq - K + \frac12  |\nabla v|^2
\end{align*}
where $K=  \overline{K} + \frac{C_0}{T}+C \sigma$.  So, either $|\nabla v|^2\leq 4K$ (and then $\max_{\overline{Q}}z\leq 2K+ C\frac\sigma 2$) or we have $- K + \frac12  |\nabla v|^2\geq \frac14 |\nabla v|^2$, and we estimate $|-\partial_t v+ |\nabla v|^2|^2\geq \frac1{16} |\nabla v|^4$. In this latter case, looking at \rife{thm:Lip_est:ineq2}   on a maximum point of $z$, where $\nabla z=0$, we deduce that
$$
 \frac \gamma{16}  |\nabla v|^4 \leq    C ( 1+ |\nabla v|^2) 
$$
which implies $|\nabla v|^2\leq C (1+ \frac1\gamma )\leq C (1+ \|u\|_\infty^2)$. Thus, we conclude in both cases with an estimate like \rife{stimaDu}, for the spatial gradient $\nabla u$. Due to Lemma \ref{lemma:u_t}, this also yields an estimate for $\partial_t u$, and then the full gradient $\overline \nabla u$ is estimated.
\end{proof}

\subsection{Local bound on the density}

We next derive a local (in time) version of the gradient estimate. This is not enough to give a local Lipschitz bound for $u$, but provides with a local $L^\infty$-bound for the density $m$.    

\begin{proposition}\label{prop:local_bound_m}
Let $u$ be a solution of \rife{elliptic:2nd_order:par}. Let  $0<a<b<T$ and $\kappa\in (0,\frac12)$.  There exists a constant $K>0$ (independent of $\rho, \de$) such that
\be\label{bound:local_m}
\vep\, (\log(m(t)) + V)+ \kappa |\nabla u|^2 \leq K\left(\frac{1}{(t-a)^2}+\frac{1}{(b-t)^2}\right) \quad\forall t\in(a,b)
\ee
where $K$ depends only on  $\kappa, T$, $||u||_{L^\infty((a,b)\times M )}$ and $\|(Ricc_g+ \nabla^2 V)_-\|_\infty$.   
\end{proposition}
\begin{proof} We follow an idea introduced in \cite{Po}, where a similar result was proved in the Euclidean case.

Let us consider $z:[a,b]\times M \to \R$ defined by
$$
z\coloneqq  \theta |\nabla u|^2-\partial_t u+\gamma\frac{u^2}{2}, \quad \hbox{where $\theta\in (0,1)$, $\gamma=\frac{1}{(1+||u||_{L^\infty((a,b)\times M )})^2}$}.
$$    
By using Lemma \ref{lemma:calc} and denoting $\varphi=\frac{|\nabla u|^2}{2}$, we obtain
\be\label{hard-loc-cafe}\begin{split}
- & tr\left(\mathcal{A}(x,\nabla u)\overline{\nabla}^2 z\right)+\vep  \nabla z\scalg\nabla V+ \rho z\leq  -2\theta\, \vep\left(|\nabla^2 u|^2+Ricc_g(\nabla u, \nabla u)+   (\nabla^2V)(\nabla u, \nabla u)\right) \\
& + 2 |\nabla (\partial_t u)|^2- 2  \nabla \varphi\scalg \nabla (\partial_t u)+ 2\theta\, [|\nabla\varphi|^2- |\nabla (\partial_tu)|^2]
- \gamma\, \mathcal{A}(x,\nabla u) (\overline{\nabla} u, \overline{\nabla} u)
\end{split}
\ee
where we notice that
\begin{align*}
&  2 |\nabla (\partial_t u)|^2- 2 \nabla \varphi\scalg \nabla (\partial_t u) + 2\theta\, [|\nabla\varphi|^2- |\nabla (\partial_tu)|^2] \\ & \quad = - \nabla (\partial_t u \scalg(  2\nabla\varphi- 2\nabla (\partial_t u) )
+ \theta \, \nabla |\nabla u|^2 \scalg ( 2\nabla\varphi- 2\nabla (\partial_t u)) -  2\theta | \nabla\varphi-  \nabla (\partial_t u) |^2 \\
& \quad = \nabla z \scalg ( 2\nabla\varphi- 2\nabla (\partial_t u) ) -  2\theta\, | \nabla\varphi-  \nabla (\partial_t u) |^2  -\gamma\, u\, 
\nabla u \scalg (  2\nabla\varphi- 2\nabla (\partial_t u))
\\
& \quad \leq  \frac1{\theta} | \nabla z|^2 + \frac1\theta\, \gamma^2\, u^2 \, | \nabla u|^2
\,.
\end{align*}
By construction of $\mathcal{A}(x,\nabla u)$, we have
$$
\mathcal{A}(x,\nabla u)( \overline{\nabla} u,\overline{\nabla} u)=||\nabla u|^2-\partial_t u|^2+\vep|\nabla u|^2
$$
so that, inserting the above estimates in \rife{hard-loc-cafe}, we obtain
\begin{align*}
- & tr\left(\mathcal{A}(x,\nabla u)\overline{\nabla}^2 z\right)+\vep  \nabla z\scalg\nabla V+ \rho z\leq  -2\theta\, \vep\left(|\nabla^2 u|^2+Ricc_g(\nabla u, \nabla u)+  (\nabla^2V)(\nabla u, \nabla u)\right) 
 \\ &  \quad +\frac1{\theta} | \nabla z|^2 + \frac1\theta\, \gamma^2\, u^2 \, | \nabla u|^2   -\gamma\, ||\nabla u|^2-\partial_t u|^2- \gamma\, \vep |\nabla u|^2\,.
 \end{align*}
Using $\gamma\, u^2\leq 1$ and the regularity of $V$ we get
\be\label{ineq:local_m}
\begin{aligned}
 & -tr\left(\mathcal{A}(x,\nabla u)\overline{\nabla}^2 z\right) + \vep \nabla z\scalg\nabla V+ \rho z 
 \leq  \left(\frac\gamma\theta + \hat{C}\right) |\nabla u|^2+\frac1\theta |\nabla z|^2  -\gamma\,  | |\nabla u|^2-\partial_t u|^2 
\end{aligned}
\ee
 where $\hat C$ is a constant depending on the lower bound of $Ric_g+ \nabla^2 V$.
Given $L>0$, let
$$
\psi\coloneqq L\left (\frac{1}{(t-a)^2}+\frac{1}{(b-t)^2}\right)
$$
Since $\psi$ blows-up at $t=a,b$, we have that $z-\psi$ admits a maximum point in $(a,b)\times M $. In such point we have $\nabla z=0$ and
$$
-tr\left(\mathcal{A}(x,\nabla u)\overline{\nabla}^2 z\right)\geq-tr\left(\mathcal{A}(x,\nabla u)\overline{\nabla}^2 \psi\right)=-6L\left (\frac{1}{(t-a)^4}+\frac{1}{(b-t)^4}\right).
$$
If $L_0\coloneqq\max(z-\psi)\geq \frac{1}{2}$, then, at a maximum point of $z-\psi$ we have
$$
|\nabla u|^2-\partial_t u=(1-\theta)|\nabla u|^2+z-\gamma\frac{u^2}{2}\geq (1-\theta) |\nabla u|^2+\psi+L_0-\frac{1}{2}\geq (1-\theta) |\nabla u|^2+\psi>0
$$
so
$$
||\nabla u|^2-\partial_t u|^2\geq (1-\theta)^2 \, |\nabla u|^4+ \psi^2.
$$
Using all of these in \eqref{ineq:local_m}, we get 
$$
-6L\left (\frac{1}{(t-a)^4}+\frac{1}{(b-t)^4}\right)\leq\,\left(\frac\gamma\theta+ \hat{C}\right) |\nabla u|^2 -\gamma\, (1-\theta)^2 |\nabla u|^4- \gamma\psi^2
$$
which gives
\begin{align*}
\left(\de\, L^2 -6L\right)\left(\frac{1}{(t-a)^4}+\frac{1}{(b-t)^4}\right) & \leq\,\left(\frac\gamma\theta+ \hat{C}\right) |\nabla u|^2 -\gamma\, (1-\theta)^2 |\nabla u|^4 
\\ & \leq K (1+ \|u\|_{L^\infty((a,b)\times M )}^2)
\end{align*}
for some $K$ only depending on $\theta$ and $\hat C$. But the inequality above cannot hold for any too large $L$ (handling constants with care, this occurs for   $L=O(K [(b-a)\vee 1]^4\, (1+ \|u\|_{L^\infty((a,b)\times M )}^2)) $). The conclusion is that we have $ \max(z-\psi)\leq \frac12$, hence
$$
z \leq L\left(\frac{1}{(t-a)^2}+\frac{1}{(b-t)^2}\right)+\frac12\,.
$$
Using the definition of   $z$   leads to 
$$
\theta\, |\nabla u|^2 - \partial_t u \leq L\left(\frac{1}{(t-a)^2}+\frac{1}{(b-t)^2}\right)+\frac12
$$
and choosing $\theta>\frac12$, from $\partial_t u = \frac12 |\nabla u|^2- \vep (\log(m)+V) $  we obtain  \rife{bound:local_m} with $\kappa=\theta-\frac12<\frac12$.
\end{proof}

\subsection{Existence of smooth solutions for a penalized problem}

We first collect all the above ingredients to show that a penalized version of the optimality system admits a classical solution. We consider, for $\de>0$, the problem
\begin{equation}
\label{2nd-order-de}
\begin{cases}
    -tr\left(\mathcal{A}(x,\nabla u)\circ\overline{\nabla}^2u\right)+ \vep \nabla u\scalg \nabla V(x)=0 &\text{in $Q_T$}\\
    -\partial_t u+\frac{1}{2}\abs{\nabla u}^2=\delta u+\epsilon(\log(m_1)+V(x)) &\text{in $\Sigma_T $}\\
    -\partial_t u+\frac{1}{2}\abs{\nabla u}^2+\delta u=\epsilon( \log(m_0)+V(x) )&\text{in $\Sigma_0 $.}\\
\end{cases}
\end{equation}
which is equivalent, reasoning as in Section \ref{1st-to-2nd-order}, to the system
 \begin{equation}
\label{mfg-de}
\begin{cases}
 - \partial_t u+\frac{1}{2}\abs{\nabla u}^2= \epsilon (\log(m)+V) &\hbox {in $Q_T $}\\
\partial_t m- div_g( m \nabla u)=0 & \hbox{in $Q_T $}\\
\de u(T)= \vep\log(m(T))-\vep\log(m_1) \,, & \hbox{in $\Sigma_T$}\\ \de u(0)= \vep\log(m_0)-\vep\log(m(0))  \, &
 \hbox{in $ \Sigma_0 $.}
\end{cases}
\end{equation} 

The auxiliary penalized problem \rife{mfg-de} has the advantage that the $L^\infty$-norm of $u$ is controlled, for $\de>0$, see \rife{deltau}. The estimates derived  from the elliptic theory (see Appendix A) yield a smooth solution, with $m>0$, provided the marginals are positive and smooth.

\begin{theorem}\label{teo-delta} Assume  that $V\in W^{2,\infty}(M)$, $m_0, m_1\in C^{1,\alpha}( M )$ and $m_0, m_1>0$ in $ M $. For every $\de>0$, there exists a unique smooth solution $(u^\de, m^\de)$ of \rife{mfg-de}, in the sense that $u\in C^{2,\alpha}(Q_T)\cap C^{1,\alpha}(\overline Q_T), m\in C^{1,\alpha}(Q_T)\cap C^{0,\alpha}(\overline Q_T)$, $m>0$ and the equations are satisfied in a classical sense.
\end{theorem}

\begin{proof} We first rely on Proposition \ref{teo-rho} which is proved in the Appendix. This gives a sequence of solutions $u_\rho $ of problem \rife{elliptic:2nd_order:app}. By maximum principle, we have $\|u_\rho\|_\infty \leq \frac C\de$. By the gradient bound proved in  Theorem \ref{thm:Lip_u_par}, we have that $ \norma{\overline{\nabla}u_\rho }_\infty\leq C$. By elliptic estimates (see also Lemma \ref{lemma:E_bounded} below) we deduce first that $\|u_\rho\|_{C^{1,\alpha}(\overline Q_T)}\leq C$, and then, bootstrapping Schauder estimates, $u_\rho$ is bounded in $C^{2,\alpha}$ on any compact subset of $  Q_T$. Defining 
$$
m_\rho= \exp\left( \frac{- \partial_t u_\rho+\frac{1}{2}\abs{\nabla u_\rho}^2}\vep -V   \right)
$$
we deduce the $C^{1,\alpha}(Q_T)\cap C^{0,\alpha}(\overline Q_T)$ estimates   on $m_\rho$ and, in particular, $m_\rho$ is uniformly bounded below due to the gradient bounds for $u_\rho$. Passing to the limit as $\rho\to 0$ gives the desired solution of \rife{2nd-order-de}, hence of \rife{mfg-de}.
\end{proof}

The next step will consist in letting  $\de\to 0$ in \rife{mfg-de}, still assuming that the marginals $m_0, m_1$ are positive, showing that the minima of \rife{func} are smooth for positive smooth marginals. This step is left to the stability results of the next Section.

\section{Existence and regularity of   optimal curves}\label{sec5}

In this section we obtain the existence and the characterization of the minima of \rife{func} in terms of the optimality system \rife{opsys}, thus proving Theorem \ref{main} stated in the Introduction. We first obtain the existence of smooth minima, whenever the marginals $m_0, m_1$ are positive and smooth; this is achieved by passing to the limit as $\de\to 0$ in problem \rife{mfg-de} and using the \lq\lq elliptic \rq\rq Lipschitz estimates of Theorem \ref{thm:Lip_u_par}. Then we 
will enlarge the set of admitted marginals $m_0$ and $m_1$ to merely $L^1$, nonnegative densities.  To this purpose, we will need a relaxed definition of weak solution to the system \rife{mfg-log}, where merely sub-solutions of the Hamilton-Jacobi equation are taken into account. This kind of notion of weak solutions, introduced in \cite{Carda1} (see also \cite{CG}, \cite{CGPT}) for first order mean-field game systems, is by now well established also in the context of mean-field transport problems,   see e.g. \cite{GMST}, \cite{ORRIERI20191868}. In particular,   in this latter paper a notion of trace was  developed for functions  which are (distributional) sub-solutions of  Hamilton-Jacobi equations, e.g. when $u$ satisfies
\be\label{subsol}
 -\partial_t u+\frac{1}{2}\abs{\nabla u}^2 \leq \alpha\,,\qquad \alpha \in L^1_{loc}((0,T)\times \Omega)
\ee
in some open set $\Omega$. Relying on the fact that $u$ is nondecreasing in time, up to an absolutely continuous function, one-sided traces of $u$ were defined in the sense of limits of measurable functions (with possible range in $[-\infty, +\infty]$), see \cite[Prop 5.6]{ORRIERI20191868}. In particular, if $\alpha\in L^1(Q_T)$,  any $u$ satisfying  \rife{subsol}
admits traces at $t=0, t=T$, denoted below as $u(0), u(T)$ respectively, which are the pointwise limits of $u(t,\cdot)$ as $t\downarrow 0$ (respectively $t\uparrow T$) in the sense of measurable functions (for a possibly well defined precise representative of $u$).

\begin{definition}\label{def:weak_sol}
A pair $(u,m)$ is a weak solution of 
\begin{equation}
\label{weak:1st_order}
\begin{cases}
    -\partial_t u+\frac{1}{2}\abs{\nabla u}^2=\epsilon (\log(m)+V(x) )\quad &\text{in $Q_T $}\\
    \partial_t m-div_g(m\nabla u)=0 &\text{in $Q_T$}\\
    m(0,\cdot)=m_0, \quad m(T,\cdot)=m_1 &\text{in $ M $}
\end{cases}
\end{equation}
if $m\in C^0([0,T];\mathcal{P}( M ))\cap L^1(Q_T)$ with $m(0)=m_0, $ $m(T)=m_1$ and $\log(m)\in L^1_{loc}((0,T);L^1( M ))$; \\ $u\in L^2_{loc}((0,T);H^1( M ))$ and in  addition $m\abs{\nabla u}^2\in L^1(Q_T)$, $m\log m\in L^1(Q_T)$ and $(u,m)$ satisfy
\begin{enumerate}[i)]
    \item $u$ is a weak sub-solution satisfying, in the sense of distributions,
    \begin{equation*}
        -\partial_t u+\frac{1}{2}\abs{\nabla u}^2\le\epsilon (\log(m)+V(x) )\quad \text{in $Q_T$}
    \end{equation*}
    \item$m$ is a weak solution satisfying, in the sense of distributions, the continuity equation
    \begin{equation*}
        \partial_t m-div_g(m\, \nabla u)=0\quad \text{in $Q_T$}
    \end{equation*}
    \item $(u,m)$ satisfy the identity
    \begin{equation}
        \label{weak_sol:cross_prod}
        \int_ M  m_0u(0)dx- \int_ M  m_1u(T)dx=\iinto \left(\frac{1}{2} \abs{\nabla u}^2m+\epsilon m(\log m+V)\right)\,dxdt
    \end{equation}
    where $u(0),u(T)$ are the one-sided traces of $u$ mentioned above.
\end{enumerate}
\end{definition}

The existence of weak solutions to \rife{weak:1st_order} will be  produced  through a stability argument, which relies on several steps. Some crucial a priori estimates are given in the following lemma, which applies to  solutions of  both \rife{mfg-de} and \rife{weak:1st_order} (corresponding to $\de=0$). 

\begin{lemma}\label{stime} Assume that $(u,m)$ is a smooth solution to \rife{mfg-de},  for some $\de\in [0,1]$, and define $\hat u:= u- \into u(T)m_1 $.  
Let $\lambda $ be given by \rife{ricci}.  There exists a constant $\hat C=\hat C( M , d,\lambda, T, \| V\|_\infty)$ such that
\be\label{int-unif-bound-u}\begin{aligned}
 & -\frac{\hat{C}}{t} \le \hat u (t,x)  \qquad \forall (t,x)\in (0,T]\times M \,, \\
 &\hat u (t,x)  \le \frac{\hat{C}}{T-t}\qquad \forall (t,x)\in [0,T)\times M 
\end{aligned}\ee
and 
\be\label{glob-ene}
\begin{split}
& \into \hat u (0)m_0\,dx \geq  -\vep \log\left(\into e^{-V}dx\right) \,
\\
&
 \iinto \left(m\, |\nabla u|^2+ \vep m\log m\right)dxdt \leq \hat C\,.
 \end{split}
\ee
In addition, for any $0<a<b<T$ there exists $C$, depending on $\hat C$ and additionally on  $a, (T-b)$, such that 
\be\label{acca-loc}
\int_a^b \!\!\!\into |\nabla u|^2\, dxdt + \vep \int_a^b\!\!\! \into |\log m|\, dxdt \leq C\,.
\ee
Finally, if $m_0,m_1>0$ on $ M $, then there exists $K$, depending on $\hat C$ and additionally on  $\min\limits_M [\vep \log(m_0e^V)]$ and $\min\limits_M [ \vep \log(m_1e^V)]$, such that 
\be \label{bound:unif_u:pos_marg}
||\hat{u}||_{L^\infty([0,T]\times M )}<K\,.
\ee
\end{lemma}

\begin{proof} By definition of $\hat u$, we have
\begin{align*}
& \int_0^T\!\!\int_ M \left(\frac{1}{2} \abs{\nabla u}^2m+\epsilon m( \log m+ V) \right)\,dxdt = \int_ M  m(0)u(0)dx-\into m(T) u(T) dx
\\ & 
\qquad  = \int_ M  (m(0)-m_0)u(0)dx-\into (m(T)-m_1) u(T)dx + \into m_0 \hat u(0)\, dx 
\end{align*}
This implies, using the initial-terminal conditions of \rife{mfg-de},
\be\label{um0sot0}
\begin{split}
& \int_ M  m_0 \hat u(0)dx =    \int_0^T\!\!\int_ M \left( \frac{1}{2} \abs{\nabla u}^2m+\epsilon m( \log m+ V)\right) \,dxdt    \\ & \quad 
+   \frac\vep\de\into (m_0-m(0))  [\log(m_0)- \log(m(0)) ] \,dx \\
& \quad + \frac\vep\de \into (m(T)-m_1) [\log m(T)-\log(m_1)]\,dx\,,
\end{split}
\ee
where last two terms should be treated as zero in case that $\de=0$. In particular, since the right-side is bounded below, we deduce that 
\be\label{um0sot}
 \begin{split}
 \int_ M  m(0) \hat u(0)dx &  \geq \int_0^T\!\!\int_ M \left(\frac{1}{2} \abs{\nabla u}^2m+\epsilon m( \log m+ V) \right)\,dxdt   \\
 & \geq   -\vep\,   \log\left(\into e^{-V}dx\right)
 \end{split}
\ee 
where we used Jensen's inequality in the last step. Of course, $\hat u$ satisfies the same Hamilton-Jacobi equation as $u$, from which we derive now all the estimates. 

At first, for any fixed $x_0\in M$,  letting $\delta_{x_0}$ be  the Dirac's delta concentrated on $x_0$,   we consider the solution $\mu$ to the continuity equation
\begin{equation*}
    \begin{cases}
        \partial_t \mu-div_g(\mu v)=0 &\text{in $(s,T)\times M $}\\
        \mu(T)=m_1, \mu(s)=\delta_{x_0} &\text{in $ M $}\\
    \end{cases}
\end{equation*}
which was built in Proposition \ref{prop:exist_comp}.  Such a curve satisfies the estimate 
\begin{equation*}
    \int^T_s\!\!\!\int_ M \left(\frac{\abs{v}^2}2\,  \mu+\epsilon \mu( \log \mu+ V) \right)dxdt\le \frac{K}{T-s}
\end{equation*}
with $K$ depending only on $M,d,\lambda, T$ and $ \|V\|_\infty $, but independent of  $x_0$.  
We multiply by $\mu$ the equation of $\hat u$  and we integrate over $(s,T)\times M $: we get
\begin{align*}
    \hat u (x_0)&= \int^T_s\!\!\!\int_ M  v\cdot \nabla u \,   \mu-\int^T_s\!\!\!\int_ M  \frac{1}{2}
    \abs{\nabla u }^2\,  \mu +\epsilon\int^T_s\!\!\!\int_ M  (\log(m ) + V)\mu \,dxdt \\
      &     \leq \int^T_s\!\!\!\int_ M  \frac{1}{2}\abs{v}^2\,   \mu + \epsilon\int^T_s\!\!\!\int_ M  (\log(m ) + V)\mu \,dxdt\,.
    \end{align*}
Using   that $a\log(b)\le a\log(a)+b$ for every $a,b\in(0,+\infty)$ we obtain
 \begin{align*}
    \hat u (x_0)& \le \int^T_s\!\!\!\int_ M  \frac{1}{2}\abs{v}^2\,   \mu+\epsilon\int^T_s\!\!\!\int_ M  \mu(\log(\mu)+ V) \,dxdt+ \vep (T-s)   \\
    &\le \frac{K}{T-s} +  \vep (T-s)  \,.
\end{align*}
So there exists a constant $\hat{C}>0$, depending on $T$, $\|V\|_\infty$ and $M$, but independent of  $\epsilon$ and $x_0$, such that for every $t\in[0,T)$ 
\begin{equation}\label{hat1}
   \hat u (t,x_0)\le \frac{\hat{C}}{T-t}\,.
\end{equation}
Reasoning in a similar way, namely using  an analogous curve  between $m_0$ and $\delta_{x_0}$, and the bound \eqref{um0sot}, we conclude that there exists a constant $\hat{C}>0$ such that for every $t\in(0,T]$
\begin{equation}\label{hat2}
    -\frac{\hat{C}}{t}\le \hat u (t,x_0).  
\end{equation}
By the arbitrariness in the choice of $x_0\in M$,  we get \eqref{int-unif-bound-u}. 

Fix now  $0<a<b<T$, from the equation of $\hat u$ we have
\begin{align*}
\int_ M  \hat u (a,\cdot)\,dx-\int_ M  \hat u(b,\cdot)\,dx & +\int_a^b\!\!\int_ M \frac{1}{2}\abs{\nabla u }^2\,dxdt 
\\  & \leq - \vep \int_a^b\!\!\int_ M |\log(m)| dxdt + \vep\, c_ M  ( 1+ \|V\|_\infty)\,.
\end{align*}
Using \rife{hat1} and   \rife{hat2}, we obtain estimate \rife{acca-loc}. Coming back to \rife{um0sot} and {  using the upper bound \rife{hat1}}, we also get \rife{glob-ene}.
%
%
%

Finally, to get \rife{bound:unif_u:pos_marg} we use the comparison principle on the equation \rife{2nd-order-de}. In fact, the linear function $\psi:=\frac {\hat C}T+Bt$ is clearly a solution of the same equation, and it is a  strict supersolution at $t=T$   if   $\epsilon(\log(m_1 e^V)) > -B$. This is possible if $m_1>0$, choosing $B$ sufficiently large. Since $\psi(0)>\hat u(0)$ by \rife{hat1}, one can compare $\hat u$ with $\psi$ (note that $\hat u$ satisfies the same condition as $u$ at $t=T$ up to a bounded term, and  $B$ can be  chosen possibly larger to compensate). Hence $\hat u \leq   \frac {\hat C}T+Bt$. Similarly we reason to get the estimate from below using the positivity of $m_0$, and we conclude with the global bound \rife{bound:unif_u:pos_marg}. 
\end{proof}

We notice that the local bounds of $\hat u$, given by \rife{int-unif-bound-u}, are totally independent of $m$. This allows us to  infer  the local  boundedness of $m$ too, as a consequence of Proposition \ref{prop:local_bound_m}.

\begin{corollary}\label{boundmloc} Under the assumptions of Lemma \ref{stime}, given any $0<a<b<T$ there exists a constant $C= C( M , d,\lambda, T, \|V\|_\infty, \|(Ric_g + D^2V)_-\|_\infty, \vep, a, T-b)$ such that
$$
\| m(t)\|_{L^\infty((a,b)\times M)} \leq C\,.
$$
\end{corollary}

\begin{proof} By estimates \rife{int-unif-bound-u}, we have $\|\hat u\|_{L^\infty((a,b)\times M)} \leq \hat C\, \max(a^{-1} , (T-b)^{-1})$. Hence, from \rife{bound:local_m} we deduce that $\vep\, \log(m(t))$ is bounded above for $t\in (a,b)$ by some constant depending on $\hat C, T, \max(a^{-1} , (T-b)^{-1})$  and $\|(Ric_g + D^2V)_-\|_\infty$, which yields the conclusion.
\end{proof}

Thanks to the global  bound \rife{bound:unif_u:pos_marg}, we can now pass to the limit as $\de\to 0$ in \rife{mfg-de}, obtaining the existence of smooth minimizers for positive smooth marginals.

\begin{theorem}\label{smoothex} Assume that $m_0, m_1\in W^{1,\infty}(M)$ and $m_0, m_1 >0$ in $M$. Let $V\in W^{2,\infty}(M)$, $\vep>0$.  Then there exists a (unique) smooth solution $(u,m)$ of  the system   \eqref{opsys} such that  $\int_ M  u(T)m_1=0$, $u\in C^{2,\alpha}(Q_T)\cap C^{1,\alpha}(\overline Q_T)$, $m \in C^{1,\alpha}(Q_T)\cap C^{0,\alpha}(\overline Q_T)$, $\alpha\in (0,1)$. In addition we have $m>0$ in $\overline Q_T$, $u$ is  a classical solution of the elliptic equation \rife{elliptic:short}, and $(m,\nabla u)$ is the unique minimum of  the functional $\mathcal{F}_\epsilon$ in \rife{func}.  

Finally, if $V\in C^{k,\alpha}(M)$, we have $u\in C^{k+1, \alpha}(M), m\in C^{k, \alpha}(M)$.
\end{theorem} 

\begin{proof} In a first step, we take  $m_0, m_1\in C^{1,\alpha}(M)$ and positive on $M$. By Theorem \ref{teo-delta}, problem \rife{mfg-de} admits a smooth solution  $(u^\de, m^\de)$, and we set as before  $\hat u^\de= u^\de- \into u^\de(T)\,m_1$. By \rife{bound:unif_u:pos_marg}, we have that $\hat u^\de$ is uniformly bounded in $Q_T$. Then, by Theorem \ref{thm:Lip_u_par}, we deduce that $\hat u^\de$ is uniformly bounded in Lipschitz norm (time-space). By elliptic estimates (same as in Lemma \ref{lemma:E_bounded}), we have that $\hat u^\de$ is bounded in $C^{1,\alpha}(\overline Q_T)$ and the bound only depends on $\| \log(m_0)\|_{W^{1,\infty}(M)},   \| \log(m_1)\|_{W^{1,\infty}(M)}$ (and of course on $\vep, \|V\|_{W^{2,\infty}}$). We also have interior local bounds on $\hat u^\de$ in $C^{2,\alpha}$, because the elliptic equation have coefficients bounded in $C^{0,\alpha}$. Therefore, by compactness, we deduce that $\hat u^\de$ converges to some $u\in C^{2,\alpha}(Q_T)\cap C^{1,\alpha}(\overline Q_T)$, which is a classical solution of the elliptic equation \rife{elliptic:short}. At the boundary, e.g. at $t=T$, we have
\be\label{bunde}
\begin{split}
 -\partial_t \hat u^\de & +\frac{1}{2}\abs{\nabla \hat u^\de}^2= \epsilon( \log(m_1)+V(x) ) + \delta \hat u^\de + \de \into u^\de(T)m_1
 \\
 & = \epsilon( \log(m_1)+V(x) ) + \delta \hat u^\de + \vep \into (\log(m^\de(T))- \log(m_1)) m_1\,.
\end{split}
\ee
However, by \rife{um0sot0} and the bound on $\hat u^\de$, we know that
$$
\into (\log(m^\de(T))- \log(m_1)) (m^\de(T)-m_1) \leq  C\, \de \,\, \mathop{\to}^{\de\to 0} 0 
$$
which implies that $m^\de(T) \to m_1$. In particular, last term in \rife{bunde} vanishes as $\de\to 0$, and since we also have $\de\hat u^\de \to 0$ we conclude that $u$ satisfies the boundary condition  at $t=T$
$$
 -\partial_t u   +\frac{1}{2}\abs{\nabla u}^2 =\epsilon( \log(m_1)+V(x) )\,.
$$
We argue in the same way for $t=0$ and we conclude that the limit $u$ satisfies the elliptic problem \rife{elliptic:2nd_order:par} with $\rho=0, \de=0$. Furthermore, we notice that $u$ satisfies an elliptic equation where the second order coefficients only depend on  $\nabla u$, and the first order coefficients depend on $\nabla V$. Hence, by the interior Schauder regularity and a boostrap argument, we have $u\in C^{k+1, \alpha}(M)$ provided $V\in C^{k,\alpha}(M)$, $k\geq 1$.

Finally, defining 
$$
m:= e^{-V(x)}\, \exp\left(\frac{-\partial_t u + \frac12 |\nabla u|^2}\vep\right)\,,
$$
we have $m\in C^{1,\alpha}(Q_T)\cap C^{0,\alpha}(\overline Q_T)$, $m>0$ and $m(0)=m_0, m(T)=m_1$. In other words, $(u,m)$ is a smooth solution of system \rife{opsys}
(unique, with the normalization of $u$), and is the  minimum of the functional $\cF_\vep$. This is also the unique minimum, as we will prove in more generality in our next results. As a last remark, the result extends by approximation to positive marginals $m_0, m_1\in W^{1,\infty}(M)$ since  in fact all the estimates above remain true.
\end{proof}

We will now extend the existence result to the case of general $L^1$ marginals. We will need some more a priori estimates and new compactness arguments. To this purpose, we first  make use  of the displacement convexity estimates  of Proposition \ref{prop:conv} to derive a local $L^2$-bound on $\nabla m$.

 \begin{lemma}\label{lemma:bound:nabla_m}
Let $(u,m)$ be a smooth solution of the system \rife{mfg-log}. For every $0<a<b<T$ there exist a constant $C=C( M,d, \lambda, T, b-a, \norma{V}_{W^{1,\infty}})$ such that
\begin{align*}
    \int_a^b\!\!\int_ M \abs{\nabla \sqrt{m}}^2\,dxdt & \le   C\, \frac{|\log \vep |}{\vep}\,.
\end{align*}
 \end{lemma}
 \begin{proof}
     First, we recall inequality \rife{serve} { which implies
          \begin{align*}
  \frac{d^2}{dt^2}\int_ M  m\log m &  \geq  -2\lambda\epsilon\int_ M  m\log m  dx
    +\frac\vep 2 \int_ M  \frac{1}{m}\abs{\nabla m}^2\,dx- C  
    \end{align*}
    for some $C$ depending on $M,d,\lambda, T, \|V\|_{W^{1,\infty}}$, and independent of $\vep \leq 1$.}
Now we fix $t_0\in(0,T)$, and $R<R_0\coloneqq\min(t_0,T-t_0)$; then, for $\tau\in(0,R)$ we let $\xi(t)$ be a smooth cut-off function such that
\begin{equation*}
\begin{cases}
    \xi(t)=1 &\text{if }t\in(t_0-\tau,t_0+\tau)\\
    \xi(t)=0  &\text{if }\abs{t-t_0}>R\\
    \abs{\xi'(t)}^2+\abs{\xi''(t)}\le \alpha_\xi
\end{cases}
\end{equation*}
for a certain $\alpha_\xi>0$. Then we have
\begin{align*}
    \frac{d^2}{dt^2}\left(\xi^2\int_ M  m\log m \,dx\right)   \ge& -2\lambda\epsilon\xi^2\int_ M  m\log m  \,dx 
     + \frac\vep 2\int_ M   \xi^2\frac{ \abs{\nabla m}^2}m \, dx - C\, \xi^2   \\
    & +4\xi\xi'\frac{d}{dt}\int_ M  m\log m \,dx+2(\xi'^2+\xi\xi'')\int_ M  m\log m \,dx \,.   
 \end{align*} 
%
%
If we integrate in $(0,T)$ both sides   we get
\begin{align*}
    \vep \iinto  \xi^2 \abs{\nabla \sqrt{m}}^2\,dxdt\le\,& C_{t_0,T} \left[ \iinto m\log m \,dxdt + C\right], 
\end{align*} 
for a possibly different constant $C$. Since $\iinto m\log(m)$ is estimated by \rife{loc_bound_entropy}, we conclude.
%
 \end{proof}

We will also use the following stability result.

\begin{lemma}\label{stab-trace}  Let $(u^n,m^n)$ be a sequence of smooth solutions of \rife{mfg-log}, possibly for different parameters $\vep_n\to\vep \geq 0$. Assume that $u_n$ satisfies \rife{int-unif-bound-u} for some constant $\hat C$ independent of $n$. Let $(u,m)$ be  such that 
$u^n\to u$ weakly in $L^2((a,b);H^1( M ))$,  for any $0<a<b< T$, and $m^n\to m$ weakly  in $L^1(Q_T)$.  Then we have:
\begin{enumerate}

\item  $u$ satisfies, in the sense of distributions,
\be\label{HJ-relax}
-\partial_t u  + \frac12 |\nabla u|^2 \leq \vep (\log m + V) \qquad \hbox{in $(0,T)\times  M $.}
\ee
\item For every sequences $m_{0n}, m_{1n}$ such that $m_{0n}\to m_0$, $m_{1n}\to m_1$ strongly in $L^1( M )$, we have
$$
\limsup_{n\to \infty} \into u^n(0)m_{0n} \leq \into u(0)dm_0 \,,
$$
$$
 \liminf_{n\to \infty} \into u^n(T)m_{1n} \geq \into u(T)dm_1\,.
$$
\item For every $(\mu, v)$ which solves \rife{conteq} and such that $\mu(t) \in L^1(M)$ for every $t$ and $\mu\log \mu \in L^1(Q_T)$, it holds
\be\label{vecchia-OPS}
\begin{split}
 \into  u(s)\mu(s) \, dx    - \into u (t)\, \mu(t)\, dx  & \leq  \int_s^t \!\!\! \into  [\mu\, v\, \scalg \nabla u  -  \frac12  |\nabla u  |^2\, \mu]\, dxd\tau
 \\ & \quad   + \vep\int_s^t \!\!\! \into  \left(    \log m    +   V \right) \mu   dxd\tau
\end{split}
\ee
for every $0\leq s<t\leq T$.
\end{enumerate}
\end{lemma} 

\begin{remark}\rm  We notice that \rife{HJ-relax} implies 
$$
-\partial_t u  + \frac12 |\nabla u|^2 \leq \vep ( m + V) \,\,\, \in L^1(Q_T)
$$
hence $u$ satisfies \rife{subsol} for some $\alpha\in L^1(Q_T)$ and admits one-sided traces up to $t=0$ and $t=T$ (those traces are used in items 2,3 above).
\end{remark}

\begin{proof} { By definition, $u^n$ satisfies
$$
 \iinto u_n\, \partial_t\vfi  \, dxdt+ \frac12 \iinto |\nabla u^n|^2\, \vfi\, dxdt= \iinto \vep (\log (m^n)+ V)\vfi\,dxdt
$$
for every $\vfi\in C^{1,0}_c( Q_T )$. In particular the integrals are restricted to some interval $(a,b)$ containing the support of $\vfi$, with $0<a<b<T$. 
Then, by simply using the weak convergence of $\nabla u_n$ in $L^2((a,b)\times  M)$ and $m^n$ in $L^1 ((a,b)\times  M)$, and the weak lower semicontinuity for  the convex functions $|p|^2$ and $-\log(m)$, respectively, we conclude that $u$ satisfies \rife{HJ-relax}, in distributional sense.}

To observe the relaxation on the initial traces, we fix $k\in\N$ and consider the truncations $u_{n,k}\coloneqq\max\{u_n,-k\}$. By a standard argument in Sobolev spaces (see e.g. Lemma 5.3 in \cite{ORRIERI20191868}), we have that  $u_{n,k}:=\max\{u_n,-k\} $ are Lipschitz functions that satisfy 
\be\label{weak:thm:trunc:1}
\begin{split}
-\partial_t u_{n,k}+ \frac12 |\nabla u_{n,k} |^2  & \leq \1_{\{u_n\ge-k\}}\epsilon_n \left(   \log m_n  +   V \right) \\
& \leq \vep_n\, (m_n+ |V|)\,.
\end{split}
\ee
 For every $k\in \N$, let $\xi_{ k}$ be the piecewise linear function  such that
    \begin{equation*}
    \begin{cases}
        \xi_{ k}(s)=1&\forall s\in\left[0,\frac{\hat{C}}{k}\right]\\
        \xi_{ k}(s)=0&\forall s\in\left[2\frac{\hat{C}}{k} ,T\right]\\
        \xi_{\delta,k}'(s)=-\frac{k}{\hat C}&\forall s\in\left[\frac{\hat{C}}{k} , 2\frac{\hat{C}}{k} \right]\\        
    \end{cases}\end{equation*}
    where $\hat{C}$ is the same constant which appears in \eqref{int-unif-bound-u}. We now fix a positive $\varphi\in L^\infty( M )$ and we  multiply  \eqref{weak:thm:trunc:1} 
    by $\phi=\xi_{ k} \varphi$.
Integrating by parts the time derivative we get
    \begin{align*}
        \int_ M  u_{n,k}(0,\cdot) \varphi dx- \frac k{\hat C} \int_{\frac{\hat{C}}{k}}^{2\frac{\hat{C}}{k} }\!\!\!\!\int_ M  u_{n,k} \varphi\,dxdt&\le C\, \frac{\|\vfi\|_\infty}k\,. 
    \end{align*}
    The sequence $u_{n,k}(0,\cdot)$ is uniformly bounded in $n$ by construction, let   $\chi_k$  be its weak-* limit, up to subsequences, in $L^\infty( M )$. By \eqref{int-unif-bound-u}, we note that $u_{n,k}=u_n$ for $t\ge \frac{\hat{C}}{k}$.  Then we pass to the limit for $n\rightarrow+\infty$, using  the weak convergence of $u_n$, and we get 
    \be\label{pre-relax}
        \int_ M  \chi_{k} \, \varphi dx- \frac k{\hat C} \int_{\frac{\hat{C}}{k}}^{2\frac{\hat{C}}{k} }\!\!\!\!\int_ M  u \, \varphi\,dxdt\le C\, \frac{\|\vfi\|_\infty}k.  
\ee
  From equation \rife{HJ-relax}, setting $F(t,x):=  \vep \int_0^t   (m(s,x)+V(x))\, ds$, we know that $u+F$ is nondecreasing in time. Hence
$$
\frac k{\hat C} \int_{\frac{\hat{C}}{k}}^{2\frac{\hat{C}}{k} }\!\!\!\!\int_ M  u \, \varphi\,dxdt \leq \int_M u( 2\frac{\hat{C}}{k}, \cdot)\vfi\, dx -    \frac k{\hat C} \int_{\frac{\hat{C}}{k}}^{2\frac{\hat{C}}{k} }\!\!\!\!\int_ M \left( F(t,x)- F(2\frac{\hat{C}}{k}, x)\right)\vfi\,dxdt
$$
and last term vanishes as $k\to \infty$ because $F$ is a primitive of a $L^1$ function. Therefore, we have
$$
\limsup_{k\to \infty}\frac k{\hat C} \int_{\frac{\hat{C}}{k}}^{2\frac{\hat{C}}{k} }\!\!\!\!\int_ M  u \, \varphi\,dxdt  \leq 
\limsup_{k\to \infty} \int_M u( 2\frac{\hat{C}}{k}, \cdot)\vfi\, dx \leq  \int_M u(0)\,\vfi\, dx
$$
where we used the pointwise convergence of $u(t, \cdot)$ as $t\to 0^+$ and the fact that $u$ is bounded above by \rife{int-unif-bound-u}, which allows us to apply Fatou's lemma. Finally, letting $k\to \infty$ in \rife{pre-relax}, we obtain 
\be\label{chik}
\limsup_{k\to \infty}  \int_ M  \chi_{k}  \varphi dx \leq 
\int_M u(0)\,\vfi\, dx
\ee
for every $\vfi\in L^\infty(M)$. Let us define now $T_j(f)= \min(f,j)$ the truncation operator in $L^1(M)$. We recall that $m_{0,n}$ converges strongly to $m_0$ in $L^1( M )$ and $u(0)$ is bounded above, so 
    \begin{align*}
        \limsup_n\int_ M  m_{0,n}u_n(0)dx & \leq    \limsup_n\int_ M  T_j(m_{0,n}) u_n(0)dx  \\ & \quad +  \limsup_n\int_ M  ( m_{0,n}-T_j(m_{0,n}))  u_n(0)dx \\
        & \le\limsup_n\int_ M  T_j(m_{0,n})u_{n,k}(0)dx  \\ & \quad + \frac {\hat C}T  \limsup_n\int_ M  (m_{0,n}-T_j(m_{0,n})) dx\\
        &\le\int_ M  T_j(m_{0})\chi_kdx + \frac {\hat C}T\int_ M  (m_{ 0}-T_j(m_{0})) dx\,.
    \end{align*}
Using \rife{chik} with $\vfi= T_j(m_0)$ we obtain, letting $k\to \infty$,
$$
\limsup_n\int_ M  m_{0,n}u_n(0)dx\leq \int_M u(0)\,T_j(m_0)\, dx +  \frac {\hat C}T\int_ M  (m_{ 0}-T_j(m_{0})) dx\,.
$$
Letting finally $j\to \infty$,  we get
$$
\limsup_n\int_ M  m_{0,n}u_n(0)dx\leq \int_M u(0)\, m_0 \, dx \,.
$$
Similarly we argue for $t=T$, using now $u_{n,k}= \min(u_n, k)$. 

We are left to prove \rife{vecchia-OPS}.  To this purpose, for any $f\in L^1(Q_T)$, we denote by $f_h, f_{-h}$ the time-average functions $f_h:= \frac1h\int_t^{t+h} f(x,s)ds$ and $f_{-h}:= \frac1h\int_{t-h}^t f(x,s)ds$.  
Integrating \rife{weak:thm:trunc:1} in $(t,t+h)$ and dividing by $h$, we obtain, by means of Jensen inequality,
$$
-\partial_t (u_{n,k})_h+ \frac12 |\nabla (u_{n,k})_h |^2 \leq \frac1h \int_t^{t+h} \left\{\1_{\{u_{n}\ge-k\}} \epsilon_n\left(   \log m_n  +   V \right)\right\}ds\,.
$$
Let $(\mu,v)$ be any solution of the continuity equation such that $\mu(t) \in L^1(M)$ for every $t$, and $\mu\log(\mu) \in L^1(Q_T)$. By a density argument, any Lipschitz function $\vfi$ can be used as test function in the continuity equation; hence, multiplying by $u_{n,k}$, we get (for any $0\leq s< r<T-h$)
\be\label{deh}\begin{split}
 & \into (u_{n,k})_h(s)\mu(s) \, dx    - \into (u_{n,k})_h(r)\, \mu(r)\, dx  \\ &  \quad \qquad 
\leq  \int_s^r\!\!\! \into  \mu\, v\, \scalg \nabla (u_{n,k})_h -  \frac12  |\nabla (u_{n,k})_h |^2\, \mu\, dxd\tau
\\ &  \quad \qquad  \quad + \vep_n \int_s^r \!\!\! \into \frac1h \int_\tau^{\tau+h} \left\{\1_{\{ u_n\ge-k\}}\left(   \log m_n  +   V \right)\right\}\mu(\tau) dsdxd\tau\,.
\end{split}
\ee
Now we let $n\to \infty$. Recall that $u_n$ is bounded in $L^2((a,b); H^1( M ))$, and, using \rife{acca-loc}, we have in fact $\partial_t u_n$   bounded in $L^1((a,b);L^1( M)) $; using   classical compactness results  (see \cite{Simon}), we deduce that $u_n$ is   compact in $L^2((a,b);L^2( M))$. Moreover, $u_n$ is locally uniformly bounded (and bounded above up to $t=0$); since $\mu(t) \in L^1(M)$, we deduce that 
$$
\into (u_{n,k})_h(t)\, \mu(t)\, dx \mathop{\to}^{n \to \infty} \into (u_{k})_h(t)\, \mu(t)\, dx 
$$
for any $t\in [0, T)$, where $u_k= \max(u,-k)$. 
Moreover, by weak convergence of $u_n$  in $L^2((a,b); H^1( M ))$, for any small $\eta>0$ we have
\begin{align*}
\limsup_{n\to \infty}   & \int_s^r \!\!\! \into  \mu\, v\, \scalg \nabla (u_{n,k})_h -  \frac12  |\nabla (u_{n,k})_h |^2\, \mu\, dxd\tau \\ & \leq 
\limsup_{n\to \infty}  \int_{s+\eta}^r \!\!\into  \mu\, v\, \scalg \nabla (u_{n,k})_h -  \frac12  |\nabla (u_{n,k})_h |^2\, \mu\, dxd\tau + \int_s^{s+\eta} \!\!\! \into\frac12 \mu |v|^2\, dxd\tau \\ &
\leq  \int_{s+\eta}^r \!\!\into  \mu\, v\, \scalg \nabla (u_{ k})_h -  \frac12  |\nabla (u_{ k})_h |^2\, \mu\, dxd\tau + \int_s^{s+\eta} \!\!\! \into\frac12 \mu |v|^2\, dxd\tau\,,
\end{align*}
 and letting $\eta \to 0$ Fatou's lemma yields
\begin{align*}  \limsup_{n \to \infty}   & \int_s^r \!\!\! \into  \mu\, v\, \scalg \nabla (u_{n,k})_h -  \frac12  |\nabla (u_{n,k})_h |^2\, \mu\, dxd\tau  
\\ & \qquad\qquad 
\leq  \int_s^r \!\!\! \into  \mu\, v\, \scalg \nabla (u_{ k})_h -  \frac12  |\nabla (u_{ k})_h |^2\, \mu\, dxd\tau\,.
 \end{align*}
 We are left with the last integral in \rife{deh}. Of course, if $\vep_n\to 0$, using $\mu \log m_n\leq \mu\log\mu + m_n$, we estimate
\begin{align*}
  & \vep_n \int_s^r \!\!\! \into \frac1h \int_\tau^{\tau+h} \left\{\1_{\{ u_n\ge-k\}}\left(   \log m_n  +   V \right)\right\}\mu(\tau) dsdxd\tau  \\ & \qquad \leq  
 \vep_n \int_s^r \!\!\! \into \frac1h \int_\tau^{\tau+h}  \left(   m_n + \mu\log \mu +   V\, \mu  \right)  dsdxd\tau \mathop{\to}^{n\to \infty} 0\,.
\end{align*}
 So we suppose 
that $\vep_n\to \vep >0$. In this case, from Lemma \ref{lemma:bound:nabla_m}  we get that  
$\sqrt{m_n}$ is bounded in $L^2((a,b); H^1( M ))$, for any $0<a<b<T$. Since   $m_n |\nabla u_n|^2$ is bounded in $L^1(Q_T)$ and  
$$
(\sqrt{m_n})_t= \frac12div_g(\sqrt{m_n}\nabla u_n)+\frac12  \nabla u_n\scalg \nabla \sqrt{m_n}
$$
we also deduce that     $(\sqrt{m_n})_t$ is bounded in $L^2((a,b); (H^1( M ))^*)+ L^1((a,b)\times  M )$. By  classical compactness results (see \cite{Simon}), we infer that $\sqrt{m_n}$ is strongly compact in $L^2((a,b)\times  M )$, which means that $m_n$ is relatively compact in the strong $L^1$-topology, locally in time, and converges almost everywhere, up to subsequences.
In particular, still using that  $\mu \log m_n\leq \mu\log\mu + m_n$, we can apply Fatou's lemma in the last integral in \rife{deh} (notice that $\1_{\{u_n\ge-k\}}$ converges to $\1_{\{u\ge-k\}}$ for almost every $k$, which we can suppose to be the case). Finally, by \rife{deh} we obtain, letting $n\to \infty$:
\begin{align*}
&   \into (u_{ k})_h(s)\mu(s) \, dx    - \into (u_{ k})_h(r)\, \mu(r)\, dx  \leq  \\
& \qquad  \leq     \int_s^r \!\!\! \into  \mu\, v\, \scalg \nabla (u_{ k})_h -  \frac12  |\nabla (u_{ k})_h |^2\, \mu\, dxd\tau \\ & \qquad 
\quad + \vep \int_s^r \!\!\! \into \frac1h \int_\tau^{\tau+h}  \1_{\{u\ge-k\}}\left(   \log m    +   V \right) \mu(\tau)    ds dxd\tau\,.
\end{align*}
Now we let $h\to 0$. Once more, last term can be handled through Fatou's lemma; indeed,  since $\log(m) \in L^1_{loc}((0,T); L^1(M))$, we have that 
$$
\frac1h \int_\tau^{\tau+h}  \1_{\{u\ge-k\}}\left(   \log m    +   V \right)  \mathop{\to}\limits^{h\to 0} \1_{\{u\ge-k\}}\left(   \log m    +   V \right)\quad \hbox{for a.e. $\tau\in (0,T), x\in M$}
$$
and
\begin{align*}
\mu(\tau) \frac1h \int_\tau^{\tau+h}  \1_{\{u\ge-k\}}\left(   \log m    +   V \right)   
\leq \mu\, \log\mu+ \mu \, V + m _h   
\end{align*}
where last sequence is strongly convergent in $L^1(Q_T)$.  Hence Fatou's lemma can be applied and yields
\begin{align*}
& \limsup_{h\to 0}\int_s^r \!\!\! \into \frac1h \int_\tau^{\tau+h}  \1_{\{u\ge-k\}}\left(   \log m    +   V \right) \mu(\tau)    ds dxd\tau  \\ & \qquad 
 \leq \int_s^r \!\!\! \into \1_{\{u\ge-k\}}\left(    \log m    +   V \right) \mu   dxd\tau\,.
\end{align*}
Similarly we argue for the term involving $\nabla u_k$, where we also use Fatou's lemma, because $\nabla u_k\in L^2_{loc}((0,T)\times M)$ and $\nabla (u_{ k})_h\to \nabla u_k$ almost everywhere, up to extracting a suitable  subsequence. 
Finally, using that $u_k$ is uniformly bounded in $(0,r+h)$ and $u$ admits one-sided traces (in the sense of monotone limits of measurable functions, as recalled above), we have that $(u_{ k})_h(t)\to u_k(t)$ (we can  use here the precise representative for $u$ at any $t$, otherwise we should limit ourselves to a.e. $t$). 
Notice that the convergence $(u_{ k})_h(t)\to u_k(t)$ is pointwise but also weak$-*$ $L^\infty$, and holds for all $t\geq 0$. Therefore, once $h\to 0$ we get
\begin{align*}
   \into  u_{ k} (s)\mu(s) \, dx - \into u_{ k} (r)\, \mu(r)\, dx   & \leq \int_s^r \!\!\! \into  [\mu\, v\, \scalg \nabla u_{ k} -  \frac12  |\nabla u_{ k} |^2\, \mu]\, dxd\tau 
 \\ & 
 + \vep\int_s^r \!\!\! \into \1_{\{u\ge-k\}}\left(    \log m    +   V \right) \mu   dxd\tau\,.
\end{align*}
Letting now $k\to \infty$, using Fatou's lemma (and the monotone convergence theorem if $s=0$), we obtain
\begin{align*}
 \into u(s)\mu(s)\, dx - \into u (r)\, \mu(r)\, dx  & \leq  \int_s^r \!\!\! \into  [\mu\, v\, \scalg \nabla u  -  \frac12  |\nabla u  |^2\, \mu]\, dxd\tau
 \\ & \quad   + \vep\int_s^r \!\!\! \into  \left(    \log m    +   V \right) \mu   dxd\tau\,.
\end{align*}
With a symmetric argument, using the left time-averages $u_{-h}$, we also obtain the inequality 
\begin{align*}
 \into  u(r)\mu(r)  \, dx    - \into u (t)\, \mu(t)\, dx  & \leq  \int_r^t \!\!\! \into  [\mu\, v\, \scalg \nabla u  -  \frac12  |\nabla u  |^2\, \mu]\, dxd\tau
 \\ & \quad  + \vep\int_r^t \!\!\! \into  \left(    \log m    +   V \right) \mu   dxd\tau\,
\end{align*}
for every $0<r<t\leq T$. Adding the last two inequalities we obtain \rife{vecchia-OPS}.
\end{proof}

\begin{remark}\rm  The inequality \rife{vecchia-OPS} includes the case that $s=0$ or $t=T$. In particular, it is a byproduct of the previous proof that $u(0)\in L^1(dm_0)$ and $u(T)\in L^1(dm_1)$, which would not be guaranteed a priori. This is indeed a consequence of item 2 of the statement, which implies that $\into u(0)m_0 \, dx $ is bounded below and $ \into u(T)m_1 \, dx$ is bounded above. Since the other bounds are obvious from \rife{int-unif-bound-u}, this yields  $u(0)\in L^1(dm_0), u(T)\in L^1(dm_1)$.
\end{remark}

\begin{remark}\label{vep0rem}\rm  We point out that inequality \rife{vecchia-OPS} remains true when $\vep=0$ even without requiring that $\mu\log(\mu)\in L^1(Q_T)$. It is enough that $\mu\in L^1(Q_T)$ in order that \rife{vecchia-OPS} holds  for all $s,t\in (0,T)$ (such that  $\mu(s), \mu(t)\in L^1(M)$), and even for $s=0, t=T$ assuming for instance that $\mu(\cdot)$ is continuous in   $[0,T]$ in the weak $L^1$-topology.

In fact, we know from Proposition \ref{prop:local_bound_m}  that $\vep_n  (\log(m_n))_+$ is locally uniformly bounded; in addition,  if $\vep_n \to 0$,  using  estimate \rife{loc_bound_entropy} we have  
\begin{align*}
\|\vep_n  (\log(m_n))_+\|_{L^1((0,T)\times M)} & \leq \vep_n \iinto m_n (\log(m_n))_+
\\
& 
 \leq C\, \vep_n   \left( 1+   |\log(\vep_n )| \right)   \to 0\,.
\end{align*}
Hence $\vep_n (\log(m_n))_+$ converges to zero in $L^1$ and weakly$-*$ in $L^\infty((a,b)\times M)$.  This implies that last term in \rife{deh} vanishes as  $n\to \infty$, only using that $\mu$ is in $L^1(Q_T)$. Thus we obtain again the inequality
$$
  \into  u_{ k} (s)\mu(s) \, dx - \into u_{ k} (r)\, \mu(r)\, dx    \leq \int_s^r \!\!\! \into  [\mu\, v\, \scalg \nabla u_{ k} -  \frac12  |\nabla u_{ k} |^2\, \mu]\, dxd\tau  
$$
for any $0<s<r<T$. Since $u$ is locally bounded, here one can  readily get rid of the truncation $k$ for $s,r\in (0,T)$. To get the inequality up to $t=0$, one can first let $s\to 0^+$  using that $\mu(\cdot)$ is continuous in the weak $L^1$-topology and $u_k$ is uniformly bounded. This leads the  first integral towards $\into  u_{ k} (0)m_0 \, dx$, which gives the desired term by letting $k\to \infty$ and using the monotone convergence theorem. Simmetrically one argue up to $t=T$ to get \rife{vecchia-OPS} in the whole interval $(0,T)$.   \qed
\end{remark}

Now we have all the ingredients to prove our  main result  on the existence and characterization of minima for nonnegative marginals $m_0, m_1$, which are  only assumed to be $L^1(M)$.

\begin{theorem}\label{weak:thm}
    Let $V\in W^{2,\infty}( M ),\,\, m_0, m_1 \in \mathcal{P}( M )\cap L^1(M)$, and $\vep >0$. 
    
    
    Then there is a unique $m$ and a unique $u$ such that $(u,m)$ is a weak solution of problem \eqref{weak:1st_order} with $\int_ M  u(T)m_1=0$.
    Moreover we have that $u,m \in L^\infty_{loc}(Q_T)$, and $(m, \nabla u)$ is the unique minimum of the functional $\cF_\vep$ in \rife{func}.
\end{theorem}
\begin{proof}  We first approximate $m_0, m_1$ with positive smooth marginals. To this goal, we consider the solutions of the heat equation 
    $$
    \frac{\partial}{\partial t}\Tilde{m}_i=\Delta \tilde m_i\quad\text{with $\Tilde{m}_i(0, \cdot)=m_i$ for $i=0,1$.}
    $$ 
    Thanks to the compactness of $ M $, it is well-known  that such solutions are smooth on $(0,+\infty)\times M $ and that they are curves of probability measures (cfr. \cite[Chapter 7]{grigoryan2009heat}). Furthermore we have $\Tilde{m}_i(t,\cdot)>0$ for every $t>0$ by the strong parabolic maximum principle (cfr. \cite[Theorem 8.11]{grigoryan2009heat}) and  $\Tilde{m}_i(t,\cdot)\to m_i$ in $L^1( M )$ for $t\to0$ (\cite[Theorem 7.19]{grigoryan2009heat}).
    
For every $n\in\N$ let $m_{0,n}=\Tilde{m}_0(\frac{1}{n},\cdot)$ and $m_{1,n}=\Tilde{m}_1(\frac{1}{n},\cdot)$. 
By Theorem \ref{smoothex}, we obtain the existence of a couple  $(u_n,m_n)$ which is a { classical  solution} of
    \begin{equation}
\label{weak:thm:0}
\begin{cases}
    -\partial_t u+\frac{1}{2}\abs{\nabla u}^2=\epsilon (\log(m)+V(x) )\quad &\text{in $Q_T $}\\
    \partial_t m-div_g(m \nabla u)=0 &\text{in $Q_T$}\\
    m(0,\cdot)=m_{0,n}, \quad m(T,\cdot)=m_{1,n} &\text{in $ M $}
\end{cases}
\end{equation}
with $\int_ M  u_n(T)m_{1,n}=0$ $\forall n\in\N$.

Now we use Lemma \ref{stime}   and  Lemma \ref{lemma:bound:nabla_m} to get estimates for $u_n, m_n$. In particular, $u_n$ satisfies \rife{int-unif-bound-u}, hence it is locally uniformly bounded, and the same holds for $m_n$ due to  Corollary \ref{boundmloc}. In addition $u_n$  is bounded in $L^2((a,b); H^1( M ))$ and 
$m_n|\nabla u_n|^2$ is bounded in $L^1(Q)$. With the same compactness arguments used in Lemma \ref{stab-trace}, we deduce that both $u_n$ and $m_n$ are  relatively compact in  the strong $L^1(M)$-topology, locally in time. 
Therefore, we deduce that there exist functions $u,m$ such that for a subsequence, not relabeled, 
\begin{align*}
u_n & \rightarrow u \quad \hbox{weakly in $L^2((a,b); H^1( M ))$ and strongly in $L^p((a,b) \times( M ))\, \forall p>1$,} \\  m_n & 
\rightarrow m \quad \hbox{strongly in $L^p((a,b) \times( M ))\, \forall p>1$, and a.e. in $Q$.}
\end{align*}
In addition, we know that $m_n(t)$ has bounded entropy (at any time $t$), so it is weakly compact in $L^1( M )$; and due to the bound of $m_n|\nabla u_n|^2$ (see \rife{glob-ene}), we get that $m_n$ is  equi-continuous from  $ [0,T]$ into the space of measures. By Ascoli-Arzela's theorem, $m_n(t)\to m(t)$ in the Wasserstein topology (uniformly in $[0,T]$), and actually in $L^1$-weak for all $t\in (0,T)$, due to the bound of the entropy. By continuity, we conclude that $m(0)=m_0$ and $m(T)=m_1$.
Notice that, by the local strong convergence of $m_n$, and due to the bound \rife{glob-ene}, we also deduce that
\be\label{root}
\sqrt{m_n}\nabla u_n \to \sqrt m \nabla u\quad \hbox{weakly in $L^2((0,T)\times M ))$,}
\ee
and, in particular, we have that $m_n\nabla u_n\to m\nabla u$ in the sense of distributions, and actually weakly in $L^1((0,T) \times M )$. Thus, we proved so far that  
 $m\in C^0([0,T]; \cP( M ) ) $ and is a weak solution of 
  \be\label{m-sol}
\begin{cases}  \partial_t m - div_g(m\nabla u)=0 & \hbox{in $Q_T $,} \\
m(0)=m_0 \,,\, m(T)=m_1 \,.& 
\end{cases}
\ee
In addition, by \rife{acca-loc} and Fatou's lemma, we have that $\log (m) \in L^1((a,b) \times( M ))$, for any $0<a<b<T$, and in particular $m>0$ a.e. in $Q$. 
  
As for the Hamilton-Jacobi equation, we use Lemma \ref{stab-trace} to deduce that 
\be\label{ueq}
-  \partial_t u + \frac12 |\nabla u|^2 \leq \vep (\log(m)+V)\,.
\ee
and   we also have
\be\label{trace0}
\limsup_{n\to \infty}  \int_ M  m_{0n} u_n(0)dx \leq  \int  u(0) dm_0\,.
\ee
In particular (since the upper bound follows from \rife{int-unif-bound-u}), we deduce that $u(0)\in L^1(dm_0)$.   
Similarly we reason for $t=T$, obtaining
\be\label{traceT}
\into m_{1}\, u(T)\, dx \leq \liminf_{n\to \infty}  \int_ M  m_{1n}u_n(T)=0\,.
\ee
As before, this implies, in particular, that  $u(T)\in L^1(dm_1)$. 
Now we only need to conclude that $(u,m)$ satisfy condition 
(iii) of Definition \ref{def:weak_sol}.  
To this purpose, we   follow the steps of \cite{ORRIERI20191868}, \cite{achdou2021mean},  on account of Lemma \ref{stab-trace}.
First we go back to \rife{weak:thm:0}, which implies, using $\into u_n(T)m_{1n}=0$,
$$
 \int_ M  m_{0n}\hat  u_n(0)dx =  \int_0^T\!\!\int_ M \frac{1}{2} m_n|\nabla u_n|^2+\epsilon m_n( \log m_n+ V)\,dxdt\,.
$$
Using \rife{root} and weak lower semicontinuity, as well as \rife{trace0}, we deduce, as $n\to \infty$:
\be\label{half1}
\int_0^T\!\!\int_ M \frac{1}{2} m |\nabla u|^2+\epsilon m ( \log m + V)\,dxdt \leq \int  u(0) dm_0\,.
\ee
However, applying \rife{vecchia-OPS}  with $(\mu,v)=(m,\nabla u)$, $s=0$ and $t=T$, we get
\begin{align*}
 \int_ M  m_{0}u(0)dx & \le\int_0^T\!\!\int_ M \frac{1}{2} \abs{\nabla u}^2 m+\epsilon m(\log m+ V)\,dxdt + \into u(T)\,dm_1 
 \\ & \leq \int_0^T\!\!\int_ M \frac{1}{2} \abs{\nabla u}^2m+\epsilon m(\log m+ V)\,dxdt 
\end{align*}
where we used \rife{traceT}. 
Putting together  the above information with \rife{half1}  we obtain
$$
\int_0^T\!\!\int_ M \frac{1}{2} m |\nabla u|^2+\epsilon m ( \log m + V)\,dxdt =  \int_ M  m_{0}u(0)dx\,,
$$
and,  in between, we also get that $\into m_1\, u(T)\, dx=0$. This means that $(u,m)$ satisfy Definition \ref{def:weak_sol}. In addition, from the bounds derived before, we have
$u,m\in L^\infty_{loc}(Q_T)$.

We are left to prove that $(m,\nabla u)$ is the unique minimum of $\cF_\vep$. 
To show that, let  
$(\mu,v)$ be an admissible couple for the functional ${\mathcal F}_\vep$.  Without loss of generality, we can assume that $\cF_\vep(\mu,v)<\infty$, in particular  $\mu\log\mu\in L^1(Q_T)$.  We use once more \rife{half1} together with \rife{vecchia-OPS} and we get
\begin{align*}
\int_0^T\!\!\int_ M & \left( \frac{1}{2} m |\nabla u|^2   +\epsilon m ( \log m + V)\right) dxdt     \leq 
\int_0^T \!\!\! \into  [\mu\, v\, \scalg \nabla u  -  \frac12  |\nabla u  |^2\, \mu]\, dxdt
 \\ & \quad  \qquad\qquad \qquad\qquad\qquad\qquad + \vep\int_0^T \!\!\! \into  \left(    \log m    +   V \right) \mu   dxdt
 \\
 & \qquad\qquad \qquad\qquad\leq \int_0^T \!\!\! \into \frac12 |v|^2\, \mu\, dxdt+ \vep\int_0^T \!\!\! \into  \left(    \log m    +   V \right) \mu   dxdt\,.
\end{align*}
By the strict convexity of $r\to r\log r -r$, for every $a\ge0$ and $b>0$ we have 
$$
a\log(a)-a\ge b\log(b)-b+\log(b) (a-b)
$$
where  $a\log a$ is extended as $0$ for $a=0$. Furthermore the equality holds if and only if $a=b$. This is equivalent to
\be\label{ineq:conv_log}
(\log(b)-\log(a))a\le b-a
\ee
with equality if and only if $a=b$. Applying this inequality with $a=\mu$ and $b=  m$ (which is positive a.e.), we obtain 
\begin{align*}
  \int_0^T\!\!\int_ M\left( \frac{1}{2} m |\nabla u|^2 +\epsilon m ( \log m + V)\right) dxdt   & 
   \leq  {\mathcal F}_\vep(\mu,v)+ \vep\int_{0}^{T}\!\!\!\!\int_ M  (m-\mu)dxdt
\\  & =\,{\mathcal F}_\vep(\mu,v)
\end{align*}
and the equality holds if and only if $\mu=m$.  This concludes the proof that $(m,\nabla u)$ is the unique minimum of $\cF_\vep$.

In fact, the solution we found is also the unique weak solution of the system \rife{weak:1st_order} (up to addition of a constant to $u$). Indeed,  the uniqueness of $(m, u)$   can be proved similarly as in \cite[Thm 1.16]{achdou2021mean}. Compared to this latter result, we observe that, being $m>0$ almost everywhere, there is no loss of information here due to  the set where $m$ vanishes.
\end{proof}

\section{Convergence to Wasserstein geodesic} 
\label{sec6}

We now briefly analyze the limit   $\epsilon\rightarrow0$ to show that the minimal curves  of   $\cF_\vep$ converge to the geodesics of the classical  mass transport problem
 \begin{equation}
         \label{functional:opt}\min \mathcal{F}_0(m,v)\coloneqq\iinto \frac{1}{2}\abs{v}^2\, dm\,, \quad(m,v):\begin{cases}\partial_t m-div_g(vm)=0\\m(0)=m_0\\m(T)=m_1\,.\end{cases}
     \end{equation}
We first consider the easier case that the marginals $m_0, m_1$ have finite entropy. This assumption implies that  $\mathcal{F}_\epsilon$ converges to  $\mathcal{F}_0$ with a rate of order $\vep$.

\begin{theorem}\label{thm:conv_vep}
Let $V\in W^{2,\infty}( M ),\,\, m_0, m_1 \in \mathcal{P}( M )\cap L^1(M)$ and such that $\cH(m_0;\nu), \cH(m_1; \nu) <\infty$. For $\vep \in (0,1)$, let $(m_\vep, u_\vep)$ be the unique solution of \eqref{weak:1st_order} given by Theorem \ref{weak:thm}, and $(m_\vep, \nabla u_\vep)$ the unique minima of $\cF_\vep$.

As $\vep \to 0$, we have:
\be\label{converge}\begin{split}
m_\vep & \to m \quad \hbox{in $C^0([0,T],\cP(M))$ and  weakly in $L^1(Q_T)$,}
\\ 
m_\vep \nabla u_\vep & \to m\nabla u\quad \hbox{weakly in $L^1(Q_T)$,}
\end{split}
\ee
where $m$ is the  Wasserstein geodesic between $m_0, m_1$, and $(m,\nabla u)$ is a minimum of $\mathcal{F}_0$.

    Moreover, we have  
$(\min \mathcal{F}_0)=\lim\limits_{\epsilon\rightarrow0} (\min \mathcal{F}_\epsilon)$, and in particular, for some $K>0$, 
\be\label{rateeps}
  | \min \mathcal{F}_\epsilon - \min \mathcal{F}_0|\leq K\, \vep\,.
\ee
\end{theorem}
\begin{proof}
For every $\vep>0$ we can apply Theorem \ref{weak:thm} and define the couple  $(u_\vep,m_\vep)$ which is the unique weak solution to the problem 
    \begin{equation*}
\label{weak:thm:1}
\begin{cases}
    -\partial_t u+\frac{1}{2}|\nabla u|^2=\epsilon (\log(m)+V(x) )\quad &\text{in $Q_T $}\\
    \partial_t m-div_g(m \nabla u)=0 &\text{in $Q_T$}\\
    m(0) =m_{0 }, \quad m(T )=m_{1 } &\text{in $ M $}
\end{cases}
\end{equation*}
with $\int_ M  u_\vep(T)m_{1}=0$.  

By Lemma  \ref{stime}, we have that  $u_\vep$ is bounded in $L^2((a,b),H^1( M ))$ and in $ L^\infty((a,b)\times M)$ for every $0<a<b<T$, so it is weakly relatively compact in $L^2((a,b); H^1( M ))$.

For any sequence extracted out of $u_\vep$, by  a diagonal process we can select (and fix) a function $u\in L^2_{loc}((0,T); H^1( M ))\cap L^\infty_{loc}((0,T)\times M)$ and a subsequence  (that we will not rename) such that $u_\vep$ converges weakly to $u$ in $L^2((a,b); H^1( M ))$ for every $0<a<b<T$. 
 
Furthermore, as a consequence of  estimate \rife{glob-ene},  we have $d_W(m_\vep(t),m_\vep (s))\leq C\sqrt{t-s}$ for any $t>s$ and some $C>0$, where $d_W$ is the Wasserstein distance.  Hence, by Ascoli-Arzela's theorem, there exists $m\in C([0,T];\cP(M))$ such that, up to a subsequence,  $m_\vep(t)\to m(t)$ in the Wasserstein topology, uniformly in time. Since $m_0, m_1$ have finite entropy, by estimate \rife{bound:semi_conv:log}, we have that $\into m_\vep(t)\log(m_\vep(t))$ is uniformly bounded in $(0,T)$.
We deduce that $m_\vep$ is weakly relatively compact in $L^1(Q_T)$ and then, by Schwartz  inequality and \rife{glob-ene}, $m_\vep \nabla u_\vep$ is also weakly relatively compact in $L^1(Q_T)$. In particular, there exists $m\in C([0,T];\cP(M))$, and $w\in L^1(Q_T)$,  such that
\begin{align*}
m_\vep &\to m \qquad \hbox{in $C([0,T];\cP(M))$ and weakly in $L^1(Q_T)$,} \\
m_\vep \nabla u_\vep  & \to w\qquad \hbox{ weakly in $L^1(Q_T)$.}
\end{align*}
We now identify $w$ as $m\nabla u$. To this goal, we first use the semi-continuity for the function  $\Psi$ defined in \rife{BB}, and we get
\be\label{rate}
\iinto \frac{|w|^2}{2m}dxdt \leq \liminf_{\vep \to 0} \iinto \frac12 m_\vep |\nabla u_\vep|^2dxdt = \into u_\vep (0)\, m_0\, dx + O(\vep)
\ee
where we used the bound on the entropy. By Lemma \ref{stab-trace} we deduce
\be\label{solita-half}
\begin{split} 
\iinto \frac{|w|^2}{2m}dxdt&  \leq \liminf_{\vep \to 0} \iinto \frac12 m_\vep |\nabla u_\vep|^2dxdt  \\ & \leq  \limsup_{\vep \to 0}  \into u_\vep (0)\, m_0\, dx
\leq \into u(0)m_0\, dx\,.
\end{split}
\ee
Notice that this inequality also implies that $w=0$ a.e. in the set $\{m=0\}$. Setting $v:= \frac{w}m \1_{\{m>0\}}$ we deduce  that $(m,v)$ is a solution to \rife{conteq}.  We also get 
from Lemma \ref{stab-trace} a similar inequality at $t=T$, namely
$$
\into u(T)m_1\, dx\leq \liminf_{\vep \to 0}  \into u_\vep (T)\, m_1\, dx=0\,.
$$
We insert this information into \rife{vecchia-OPS} (where $\vep=0, s=0, t=T$) and we get
$$
\into u(0)m_0\, dx \leq \iinto [\frac wm\scalg \nabla u-\frac{| \nabla u|^2}2] \,m\, dx\leq \iinto \frac{|w|^2}{2m}\, dxdt\,. 
$$
Combining  this with  \rife{solita-half} we conclude that $w= m \, \nabla u$ and that
$$
\into u(0)m_0\, dx =  \iinto \frac12 m  |\nabla u |^2dxdt \,, 
$$
and then 
$$
 \into u_\vep (0)\, m_0\, dx \to \into u(0)m_0\, dx\,,
 $$
 and
 $$
  \iinto \frac12 m_\vep |\nabla u_\vep|^2dxdt\to \iinto \frac12 m  |\nabla u |^2dxdt\,.
$$
We now show that $(m,\nabla u)$ is a minimum of $\cF_0$. To this goal, we recall (see \cite{Erasquin-McCann}, \cite{McCann2}) that the minimum of $\cF_0$ is attained at a unique geodesic $\mu^*$ such that $\mu^*(t)\in L^1(M)$ for every $t$ and $\mu^*(\cdot)$ is   continuous in the weak-$L^1$ topology. On account of Remark \ref{vep0rem}, we can use inequality  \rife{vecchia-OPS} with $\vep =0$ and $\mu=\mu^*$, which yields:
\begin{align*}
\iinto \frac12 m  |\nabla u |^2dxdt & = \into u(0)m_0\, dx \leq \int_0^T \!\!\! \into  [\mu^*\, v\, \scalg \nabla u  -  \frac12  |\nabla u  |^2\, \mu^*]\, dxdt
\\ & \leq  \int_0^T \!\!\! \into  \frac12 |v|^2\, \mu^*\, dxdt = \min \mathcal{F}_0
\end{align*}
Hence $(m,\nabla u)$ is the minimum point of $\cF_0$ and coincides with the unique geodesic between $m_0, m_1$ (notice that the previous inequality implies $v=\nabla u$). 
Finally, we observe that, still using \rife{vecchia-OPS} for $(m_\vep, u_\vep)$ and $(m,u)$, we have
\begin{align*}
 \min \mathcal{F}_\epsilon   
 & = \into u_\vep (0)\, m_0\, dx 
\\ & \leq   \iinto \frac12 m  |\nabla u |^2dxdt  + \vep \iinto m(\log(m_\vep)+ V) \, dxdt
\\ & \leq  \min \mathcal{F}_0  + \vep \iinto m\log (m) + C\, \vep =  \min \mathcal{F}_0 + O(\vep)
\end{align*}
due to the fact that $m$ has finite entropy. Since the opposite inequality is obviously true, we conclude  with \rife{rateeps}.
\end{proof}

As a byproduct of the previous result, we have proved that whenever $m_0, m_1$ have finite entropy, then the Wasserstein geodesic connecting $m_0, m_1$ has finite entropy for all times. In fact, using Corollary \ref{corollary:semi_conv:log} we also have a quantitative estimate of the semiconvexity of the  log-entropy functional  along the geodesic.
In this way, we recover a result proved by Daneri and Savar\'e (\cite{Daneri_2008}) with a different and independent approach. We point out that we can avoid the use of this semiconvexity property of the geodesic in all our estimates and stability results (see also Remark \ref{dasoli}), so that  this is just  deduced from the limit $\vep\to0$ of the semiconvexity of the optimal curves of $\cF_\vep$. 

\begin{corollary}
    In the assumptions of Theorem \ref{thm:conv_vep}, let $\Lambda\in \R$ be such that
    $$
    Ric_g + D^2 V \geq \Lambda \, I_g 
    $$
    in the sense that $(Ricc_g+ D^2 V)(X,X)\ge  \Lambda \abs{X}^2$ for every vector field $X$ on $M$.
    
    Then the relative entropy functional $\cH(m ; \nu)$ is $\Lambda$-convex along the 2-Wasserstein geodesics. In other words, if $m:[0,1]\to\cP( M )$ is the  geodesic between $m_0$ and $m_1$, it holds
\be\label{saragiu}
   \cH(m(t) ; \nu)\leq t \cH(m_1 ; \nu)+(1-t)\cH(m_0 ; \nu)-\frac\Lambda2 t(1-t)W_2^2(m_0,m_1) 
\ee
for every $t\in[0,1]$.
\end{corollary}
\begin{proof}
As in Theorem \ref{thm:conv_vep},  let $(u_\vep, m_\vep)$ be the sequence of solutions of \rife{opsys}, with $\int_ M  u_\vep(T)m_{1}=0$. At first, let us suppose that $m_0, m_1$ are smooth and positive, so that $u_\vep, m_\vep$ are smooth solutions (Theorem \ref{smoothex}). By \rife{BK+displa} in Corollary \ref{corollary:semi_conv:log}, we have
    \begin{align*}
    \frac{d^2}{dt^2} \cH(m_\vep(t); \nu) & \geq \Lambda \into m_\vep\, |\nabla u_\vep |^2dx  \\
    & = 2\Lambda {\mathcal B}_\vep(m_0, m_1) + 2\Lambda \vep \cH(m_\vep(t); \nu)
     \end{align*}
     where we used \rife{bound:E} (with $T=1$) and, we recall,  ${\mathcal B}_\vep(m_0, m_1)= \min(\cF_\vep)$.  Since $m_0, m_1$ have finite entropy, we already know from \rife{bound:semi_conv:log} that $ \cH(m_\vep(t); \nu)$ is bounded independently of $\vep$, for every $t\in (0,T)$. Hence we get, for some $C>0$:
     $$
      \frac{d^2}{dt^2} \cH(m_\vep(t); \nu) \geq \Lambda_\vep:= 2\Lambda {\mathcal B}_\vep(m_0, m_1) - C\, \vep\,
     $$ 
    which implies
\be\label{saragiu2}
     \cH(m_\vep(t) ; \nu)\leq t \cH(m_1 ; \nu)+(1-t)\cH(m_0 ; \nu)- \frac{\Lambda_\vep}2 t(1-t) \quad\forall t\in([0,1].
\ee
Note that $\Lambda_\vep$ is stable by approximation of $m_0, m_1$ with smooth positive densities, due to Theorem \ref{weak:thm}, so the above inequality holds for any $m_0, m_1$ with finite entropy. Finally, we let $\vep \to 0$; by Theorem \ref{thm:conv_vep} we know that $m_\vep(t)$ weakly converges in $L^1(M)$ towards $m(t)$, where $m$ is the Wasserstein geodesic between $m_0, m_1$.  Thus we can use the weak lower semicontinuity of the entropy for the left-hand side. We also know that  $ {\mathcal B}_\vep(m_0, m_1)= \min(\cF_\vep)$ converges towards $ \min(\cF_0)=\frac12 W_2^2(m_0,m_1)$.  Hence $\Lambda_\vep \to      \Lambda \, W_2^2(m_0,m_1)$ and from \rife{saragiu2} we deduce \rife{saragiu}.
\end{proof}

We conclude by extending the convergence result of Theorem \ref{thm:conv_vep} to  marginals which only belong to $L^1(M)$, without having necessarily finite entropy.
It is known that, if $m_0, m_1\in L^1(M)$, then the Wasserstein geodesic belongs to $L^1(M)$ for all times $t$, see e.g. \cite{Erasquin-McCann}, \cite{McCann2}. This is also a byproduct of our next result, since we will prove that \rife{converge} still holds for merely $L^1$ marginals.  To get the necessary equi-integrability for this goal,  we will use an idea suggested to us by G. Savar\'e, based  on displacement convexity and the following lemma essentially contained in \cite{Villani-oldnew}.

\begin{lemma}\label{giuseppe-villani} Let $m_0\in L^1(M)$. Then there exists a function $U: [0,\infty)\to \R^+$ such that:

(i)  $U\in C^2(0,\infty)$,  is convex and satisfies $\frac{U(r)}r \mathop{\to}\limits^{r\to \infty} +\infty$.

(ii) $ P'(r)r - (1-\frac1d) P(r) \geq 0$, and $ P(r) \leq K\, r $ for every $ r>0$, where $P(r)=U'(r)r-U(r)$, $K>0$.

(iii) $U(m_0) \in L^1(M)$.
\end{lemma}

Even if the above Lemma  mostly  follows from  \cite[Proposition 17.7]{Villani-oldnew} combined with De la Vall\'ee Poussin lemma, we will give a proof in the Appendix, for the reader's convenience.  Standing on Lemma \ref{giuseppe-villani}, we will show that $ \min\mathcal{F}_{\epsilon}(m_0,m_1) \to  \min\mathcal{F}_0(m_0, m_1)$, although now the rate of convergence appears to be of order $O(\vep \log \vep)$. We will also prove a  further property here, namely that, up to approximating $m_0, m_1$ with suitable smooth sequences $m_{0\vep}, m_{1\vep}$, we can build a minimizing curve of $\mathcal{F}_{\epsilon}$ which is a smooth approximation of the Wasserstein geodesic, with the adjoint states uniformly converging to the Kantorovich potentials. 

\begin{theorem}\label{conv_vep2}
Let $V\in W^{2,\infty}( M )$, and $m_0, m_1 \in \mathcal{P}( M )\cap L^1(M)$. For $\vep \in (0,1)$, let $(m_\vep, u_\vep)$ be the unique solution of \eqref{weak:1st_order} given by Theorem \ref{weak:thm}, and  $m$ be  the   Wasserstein geodesic between $m_0, m_1$, with $(m,\nabla u)$ the  minimum of $\mathcal{F}_0$.
Then we have that  \rife{converge} holds true and  $\min \mathcal{F}_0=\lim\limits_{\epsilon\rightarrow0} \min \mathcal{F}_\epsilon$.
In particular, we have   
\be\label{rateeps2}
\min \mathcal{F}_0-c_0 \vep \leq \min \mathcal{F}_\epsilon \leq  \min \mathcal{F}_0 + c_1 \, \vep |\log \vep |\,
\ee
for some $c_0, c_1>0$.

In addition, there exist   sequences $m_{0\vep}, m_{1\vep}$, converging respectively to $m_0, m_1$ in $L^1(M)$, such that:

(i) $(m_\vep, \nabla u_\vep):= {\rm argmin} \mathcal{F}_\epsilon $ is smooth in $(0,T)\times M$.

(ii) $u_\vep$ is bounded in $W^{1,\infty}(Q_T)$   and converges uniformly to a   Lipschitz continuous solution $u$ of the Hamilton-Jacobi equation $\partial_t u = |\nabla u |^2/2$ in $Q_T$.

(iii) $m_\vep \to m$ in $C^0([0,T],\cP(M))$ where $m$ is the Wasserstein geodesic connecting $m_0, m_1$, 
with $\nabla u$ being the metric velocity of the geodesic and $u(0), u(T)$   the Kantorovich optimal potentials. 
\end{theorem}

\begin{proof}  Let $U$ be the function given by Lemma \ref{giuseppe-villani} (built replacing $m_0$ with $\max(m_0, m_1)$). Using Proposition \ref{prop:conv}, we have
\begin{align*}
  \frac{d^2}{dt^2}\int_ M  U(m_\vep)\,dx & \ge  \int_ M  P(m_\vep)Ricc_g(\nabla u_\vep,\nabla u_\vep)\,dx -  \vep \int_ M  m_\vep \,  |  \nabla V|^2\,dx 
\\ & \geq -   \lambda \, K\int_ M  m_\vep\,  |\nabla u_\vep|^2 \,dx- c\, \vep\,
\end{align*}
where we used the property (ii) of $U$,  from Lemma \ref{giuseppe-villani}. Setting $\vfi_\vep(t)= \int_ M  U(m_\vep)(t)\,dx$, we deduce that $\vfi_\vep$ satisfies
$$
\begin{cases} -\vfi_\vep''(t) \leq \lambda\, K  f_\vep(t) + c\vep & t\in (0,T)
\\
\vfi_\vep(0) = \int_M U(m_0)\,,  \vfi_\vep(T) = \int_M U(m_1) & \,,
\end{cases}
$$
where $f_\vep:= \int_ M  m_\vep\,  |\nabla u_\vep|^2 \,dx$ is bounded in $L^1(0,T)$ by Lemma \ref{stime}. By the (compact) embedding of $H^1(0,T)$ into $C^0([0,T])$, we deduce that $\vfi_\vep$ is uniformly bounded and actually it is relatively compact in the uniform topology. Since $U$ is superlinear, this means that $m_\vep(t)$ is weakly compact in $L^1(M)$ and weakly converges to $m(t)$, for every $t\in [0,T]$. In addition, $t\mapsto m(t)$ is continuous in $L^1(M)$ endowed with the weak topology.  With this in hand, we also have that $m_\vep\nabla u_\vep$ is weakly compact in $L^1(Q_T)$.  Moreover, from Lemma \ref{stime}, we know that $u_\vep$ weakly converges to some $u\in L^2_{loc}((0,T); H^1( M ))\cap L^\infty_{loc}((0,T)\times M)$, exactly as in Theorem \ref{thm:conv_vep}.  In order to identify the limit of $m_\vep\nabla u_\vep$ as $m\nabla u$, we can  proceed as before,  using \rife{vecchia-OPS} on account of Remark \ref{vep0rem}. 
Thus we obtain the same conclusion \rife{converge} as before. 
However, the rate of convergence \rife{rateeps} does not follow in this case since we can only estimate
$$
\iinto m_\vep \log(m_\vep) \,dxdt = O (\vep|\log\vep|)
$$
from estimate \rife{loc_bound_entropy}. This yields \rife{rateeps2}.

Finally, we observe that we can build a smooth approximation of the Wasserstein geodesic if we approximate $m_0, m_1$ with the heat semigroup, namely  $m_{0\vep}= S_\vep(m_0)$, $m_{1\vep}= S_\vep(m_1)$, where $S_t$ is the heat semigroup as in Proposition \ref{prop:exist_comp}. By using Li-Yau estimates on the heat kernel, in the improved form given e.g. in \cite[Thm 1.8]{Li-Zhang} for Riemannian manifolds with Ricci curvature bounded below, we have that there exists a constant $C$, only depending on $M,d$, such that
$$
\vep \,\left( \frac{|\nabla S_\vep(m_0)|}{S_\vep(m_0)}+ |\log(S_\vep(m_0))|\right) \leq C\,.
$$ 
This means that $\vep \log(m_{0\vep})$ is bounded in $W^{1,\infty}(M)$, and so is for $\vep \log(m_{1\vep})$. From \rife{bound:unif_u:pos_marg} in Lemma \ref{stime} we deduce that $u_\vep$ is uniformly bounded, and then from Theorem \ref{thm:Lip_u_par} $u_\vep$ is bounded in Lipschitz norm. Moreover, at fixed $\vep$, $(u_\vep, m_\vep)$ are smooth according to Theorem \ref{smoothex}.  Finally, by Ascoli-Arzel\'a's theorem,  $u_\vep$  is relatively compact in $C^0(\overline Q_T)$ and converges uniformly towards its limit $u$. It is a standard result that $u$ is a viscosity solution of the Hamilton Jacobi equation $\partial_t u=\frac12 |\nabla u|^2$.
\end{proof}

{ \section{Appendix A: existence of smooth solutions}}
Here we show the existence of solutions to  the differential system
\begin{equation}
\label{elliptic:2nd_order:app}
\begin{cases}
    -tr\left(\mathcal{A}(x,\nabla u)\circ\overline{\nabla}^2u\right)+\rho u+\vep \nabla u\scalg \nabla V(x)=0 &\text{in $Q$}\\
    -\partial_t u+\frac{1}{2}\abs{\nabla u}^2=\delta u+\vep (\log(m_1)+V(x)) &\text{in $\Sigma_T$}\\
    -\partial_t u+\frac{1}{2}\abs{\nabla u}^2+\delta u= \vep( \log(m_0)+V(x) )&\text{in $\Sigma_0$}\\
\end{cases}
\end{equation}
and we recall (see \rife{elliptic:long}) that the expanded form of the first equation is
\begin{multline}
\label{elliptic:long:app}
    -\partial_{tt} u+ 2 \nabla (\partial_t u) \scalg \nabla u-\vep\Delta_g u-  (\nabla^2u)(\nabla u, \nabla u)
    +\rho u+\vep \nabla u\scalg \nabla V=0.
\end{multline}

\begin{proposition}\label{teo-rho}Let $\rho,\delta, \vep>0$. Assume that $V\in W^{2,\infty}(M)$ and $m_0,m_1\in\cP(M)\cap C^{1,\alpha}(M)$ with $m_0,m_1>0$ in $M$.\\
Then there exists a  classical solution $u\in C^{2,\alpha}(\overline Q_T)$ of the quasilinear elliptic problem \rife{elliptic:2nd_order:app}.
\end{proposition}

We will prove  such result by means of the method of continuity. For convenience of notation, we set $\vep=1$ in what follows. We consider the differential operators  $F:C^{2,\alpha}(\overline Q_T)\rightarrow C^\alpha(\overline Q_T)$ and $G:C^{2,\alpha}(\overline Q_T)\rightarrow C^{1,\alpha} (\Sigma_0\cup\Sigma_T)$ defined by
\begin{align*}
    F[u]&\coloneqq-tr\left(\mathcal{A}(x,\nabla u)\circ\overline{\nabla}^2u\right)+\rho u+\nabla u\scalg \nabla V(x)\\
    G[u]&\coloneqq\begin{cases}-\partial_t u+\frac{1}{2}\abs{\nabla u}^2-\delta u- \log(m_1)-V(x) &\text{in $\Sigma_T$}\\
    -\partial_t u+\frac{1}{2}\abs{\nabla u}^2+\delta u- \log(m_0)-V(x) &\text{in $\Sigma_0$.}\\       
    \end{cases}
\end{align*}
Finally we define the operator
\begin{equation*}    
\begin{array}{rcl}
    P:C^{2,\alpha}(\overline Q_T) & \longrightarrow & C^\alpha(\overline Q_T)\times C^{1,\alpha} (\Sigma_0\cup\Sigma_T)\\
    u &\longrightarrow &\left(F[u],G[u]\right) 
\end{array}\end{equation*}
and the set 
\begin{equation*}
    E\coloneqq\Set{u\in C^{2,\alpha}(\overline Q_T)|\exists \tau\in[0,1]\quad s.t.\quad P[u]=(1-\tau)P[0]}
\end{equation*}
We note that $0\in E$ and that if a function $u\in C^{2,\alpha}(\overline Q_T)$ satisfies $P[u]=0$ then it is a solution for the elliptic problem \eqref{elliptic:2nd_order:app}.
{ \begin{lemma}
\label{lemma:E_bounded}
    The set $E$ is bounded in $C^{1,\alpha}(\overline Q_T)$.
\end{lemma}
\begin{proof}
    In the expanded form, given $\tau\in [0,1]$, the problem $P[u]=(1-\tau)P[0]$ is
    \begin{equation}
\begin{cases}
\label{app:syst:tau:NL}
    -tr\left(\mathcal{A}(x,\nabla u)\circ\overline{\nabla}^2u\right)+\rho u+\nabla u\scalg \nabla V(x)=0 &\text{in $Q_T$}\\
    -\partial_t u+\frac{1}{2}\abs{\nabla u}^2=\delta u+\tau\left( \log(m_1)+V(x)\right) &\text{in $\Sigma_T$}\\
    -\partial_t u+\frac{1}{2}\abs{\nabla u}^2+\delta u= \tau\left( \log(m_0)+V(x)\right) &\text{in $\Sigma_0$}\\
\end{cases}
\end{equation}
 So by Theorem \ref{thm:Lip_u_par}, \rife{deltau}, and $\tau\le1$ we get
\begin{align*}
    \norma{u}_{W^{1,\infty}}&\le C( \de,\norma{\tau  \log(m_0)}_{W^{1,\infty}},\norma{\tau  \log(m_1)}_{W^{1,\infty}},\norma{\tau V}_{W^{2,\infty}})\\
    &\le C(\de,\norma{ \log(m_0)}_{W^{1,\infty}},\norma{  \log(m_1)}_{W^{1,\infty}},\norma{V}_{W^{2,\infty}})\,.
\end{align*}
Moreover, since $V\in W^{2,\infty}$, we get that the coefficients of the elliptic problem $P[u]=(1-\tau)P[0]$ are $C^{0,\alpha}$ in the interior and $C^{1,\alpha}$ on the boundary with respect to $(t,x)$,  independently of $\tau$.\\
Fixed any local system of coordinates on $ M $, we recall that the second covariant derivatives of $\psi$ are
\begin{equation*}
    \nabla_{ij}\psi=\frac{\partial^2}{\partial x_i\partial x_j}\psi-\sum_{k}\Gamma_{ij}^k\frac{\partial}{\partial x_k}\psi
\end{equation*}
where the $\Gamma_{ij}^k$ are the Christoffel symbols.

Hence, if we localize the differential system \eqref{app:syst:tau:NL}, we get a differential problem on $\R^d$ which differs from \eqref{app:syst:tau:NL} only in the first order terms (because of the Christoffel symbols, which are smooth), so an elliptic problem with H\"older coefficients  in their arguments. Therefore, we can apply the classical results on Schauder estimates (e.g. \cite{gilbarg2001elliptic}, and \cite[Lemma 2.4]{lieb84NA} for the boundary estimates) 
in every local chart of a finite atlas ($ M $ is compact).
We obtain a global $C^{1,\alpha}(\overline Q_T)$ estimate on $u$, independent of $\tau$.
\end{proof}
}
We now observe that, thanks to Lemma 17.29 in \cite{gilbarg2001elliptic},  we have that $E$ is closed in $C^{2,\alpha}(\overline Q_T)$ and that the set
\begin{equation*}
S\coloneqq\Set{\tau\in[0,1]|\exists u\in C^{2,\alpha}(\overline Q_T)\quad s.t.\quad P[u]=(1-\tau)P[0]}    
\end{equation*}
is closed. Note that $S$ is not empty because $0\in S$.
We want to prove that $S$ is also open, so that it will coincide with $[0,1]$. To this purpose, we  prove that  for each $u\in E$ the linear differential system induced by the Gateaux derivative of $P$ has one and only one solution.

\begin{lemma}
Given $u\in E$, $\phi\in C^{0,\alpha}(\overline Q_T), \zeta_1,\zeta_0\in { C^{1,\alpha}( M )}$ there exists one and only one solution $\psi\in C^{2,\alpha}(\overline Q_T)$ of the linear problem 
\begin{equation}
\label{app:linear_diff_prob}
    \begin{cases}
        \begin{aligned}
            -\partial_{tt}\psi -  \Delta_g \psi&-(\nabla^2\psi)(\nabla u, \nabla u)+2\nabla (\partial_t\psi) \scalg \nabla u+\\
            &+[2\nabla (\partial_tu) -\nabla \left(|\nabla u|^2\right)+\nabla V]\scalg \nabla \psi+\rho\psi=\phi
        \end{aligned} & \text{in $Q_T$}\\
        -\partial_t \psi+2\nabla u\scalg \nabla \psi-\delta \psi=\zeta_1 &\text{on $\Sigma_T$}\\
        -\partial_t\psi +2\nabla u\scalg \nabla \psi+\delta \psi=\zeta_0 &\text{on $\Sigma_0$.}\\
    \end{cases}
\end{equation}
\end{lemma}
\begin{proof}
The uniqueness of the solution follows from the maximum principle applied to the homogeneus system ($\zeta_1=\zeta_0=\phi=0$), thanks to the linearity of the problem.

We now prove the existence of a solution. Let $L:C^{2,\alpha}(\overline Q)\rightarrow C^{0,\alpha}(\overline Q)\times { C^{1,\alpha}( M )^2} $ defined by the left side of system \eqref{app:linear_diff_prob}, in other words a function $\psi\in C^{2,\alpha}$ solves \eqref{app:linear_diff_prob} if and only if $L[\psi]=(\phi,\zeta_1,\zeta_0).$

Let $L_0:C^{2,\alpha}(\overline Q)\rightarrow C^{0,\alpha}(\overline Q)\times { C^{1,\alpha}( M )^2}$ be the operator given by 
\begin{equation*}
    L_0[\psi]=(-\partial_{tt}\psi -\Delta_g\psi+\rho \psi,-\delta\psi,\delta\psi)
\end{equation*}
and, for every $\tau\in[0,1]$,
\begin{equation*}    
\begin{array}{rcl}
    L_\tau:C^{2,\alpha}(Q) & \longrightarrow & C^\alpha(Q)\times {C^{1,\alpha}}( M )^2\\
    \psi &\longrightarrow &\tau L[\psi]+(1-\tau)L_0[\psi]. 
\end{array}\end{equation*}
so the expanded form of the problem $L_\tau[\psi]=(\phi,\zeta_1,\zeta_0)$ is
\begin{equation}
    \label{app:syst:tau}\begin{cases}
        \begin{aligned}
            -\partial_{tt}\psi-&\Delta_g\psi -\tau(\nabla^2\psi)(\nabla u, \nabla u)+2\tau \nabla (\partial_t\psi)\scalg \nabla u+\\
            &+\tau[2\nabla (\partial_t u)-\nabla \left(|\nabla u|^2\right)+\nabla V]\scalg \nabla \psi+\rho\psi=\phi
        \end{aligned} & \text{in Q}\\
        -\tau\partial_t\psi+2\tau \nabla u\scalg \nabla \psi-\delta \psi=\zeta_1 &\text{on $\Sigma_T$}\\
        -\tau\partial_t\psi+2\tau \nabla u\scalg \nabla \psi+\delta \psi=\zeta_0 &\text{on $\Sigma_0$.}\\
    \end{cases}
\end{equation}
As in Lemma \ref{lemma:E_bounded}, we observe that 
if we write   system \eqref{app:syst:tau} in the coordinates of  any local chart, we obtain a uniformly  elliptic problem.
Hence, on every differential problem induced by the localizations of a finite atlas ($ M $ is compact), we can, by ellipticity, apply Theorem  6.17 or Theorem 6.30 of \cite{gilbarg2001elliptic} (the former for $t\in(0,T)$, the latter for $t\in\{0,T\} $). We obtain the global bound
\begin{equation*}
    \norma{\psi}_{C^{2,\alpha}(  \overline Q )}\le C(\norma{\psi}_{C^{\alpha}(\overline Q )}+\norma{\phi}_{C^{\alpha}(\overline Q )}+\norma{\zeta_1}_{{ C^{1,\alpha}}( M )}+\norma{\zeta_0}_{{ C^{1,\alpha}}( M )})
\end{equation*}
where $C$ is independent from $\tau$.
%
We conclude that there exists a constant $C>0$, depending only on $\rho,\delta$ and $ M $, such that
\begin{equation*}
    \norma{\psi}_{C^{2,\alpha}(\overline Q)}\le C\norma{L_\tau[\psi]}_{C^{0,\alpha}(\overline Q)\times { C^{1,\alpha}( M )^2} }\quad\forall\tau\in[0,1],\,\forall\psi\in C^{2,\alpha}(Q)\,.
\end{equation*}
Thanks to the bound above, we can apply the method of continuity for linear elliptic differential problems and we find that, for each $\tau\in[0,1]$, the operator $L_\tau$ is surjective if and only if $L_0$ is surjective. This means that the system \eqref{app:linear_diff_prob} has solution if and only if there is a solution to the system
\begin{equation*}
    \begin{cases}
    -\partial_{tt}\psi-\Delta_g\psi+\rho\psi=\phi & \text{in Q}\\
    -\delta\psi=\zeta_1 & \text{on $\Sigma_T$}\\
    \delta\psi=\zeta_0 & \text{on $\Sigma_0$}.\\
    \end{cases}
\end{equation*}
This   is a standard Poisson problem, which admits a  unique solution since the differential operator $-\frac{\partial^2}{\partial^2 t}-\Delta_g+\rho$ is Fredholm of index zero and injective (see Proposition 1.8, Chapter 5 of \cite{taylor2010partial}). Moreover, by elliptic regularity,  the solution belongs to $C^{2,\alpha}(\overline Q)$.
\end{proof}

\section{Appendix B: proof of Lemma \ref{giuseppe-villani}.}

In this section we provide the proof of Lemma \ref{giuseppe-villani}, which we restate here for the reader's convenience.

\begin{lemma} Let $m_0\in L^1(M)$. Then there exists a function $U: [0,\infty)\to \R^+$ such that:

(i)  $U\in C^2(0,\infty)$,  is convex and satisfies $\frac{U(r)}r \mathop{\to}\limits^{r\to \infty} +\infty$.

(ii) $ P'(r)r - (1-\frac1d) P(r) \geq 0$, and $ P(r) \leq K\, r $  for every $ r>0$, and for some $K>0$, where $P(r)=U'(r)r-U(r)$.

(iii) $U(m_0) \in L^1(M)$.
\end{lemma}

\begin{proof} 
The proof consists of two steps, which rely on  two general facts: first of all, by the classical De la Vall\'ee Poussin criterion, any $m_0 \in L^1(M)$ is integrable for some 
superlinear function $\Psi$. Secondly, by \cite[Proposition 17.7]{Villani-oldnew}, given a superlinear function $\Psi$ one can build a new function $U$ which stays below and satisfies the conditions needed for the displacement convexity. However, for our purposes, it will be needed to detail the possible construction of $\Psi$  as well as  Villani's construction of $U$, in order to show that $P$ satisfies a suitable sublinear growth. This is why we detail the proof, for the reader's convenience.

{\bf Step 1.} \quad Here we follow and slightly modify, to our purposes, the proof of de la Vall\'ee Poussin's lemma proposed in \cite{canizo}. If we set $f(\lambda):={\rm meas}(\{ m_0>\lambda\})$, then $f\in L^1(0,\infty)$ and we can define the sequence
\be\label{an}
a_n:= \inf\{ t\,: \int_t^\infty f(\lambda) d\lambda <\frac1{n^2}\}\,.
\ee
We possibly increase $a_n$ by defining the sequence $x_n$ as 
$$
\begin{cases}
x_1=1 & \\
x_{n+1}= \max(2x_n, a_{n+1}+1) & 
\end{cases}
$$
Now we define  a piecewise linear function (which will play the role of $\Psi'$) as follows:
$$
\begin{cases}
\Phi(x)=x & x\in (0,1) \\
\Phi(x)= \Phi(x_n) + d_{n+1}(x-x_{n}) & \hbox{for $x\in (x_n, x_{n+1})$}\,\, n\geq 1 \end{cases}  
$$
where $d_n$ is defined as
$$
\begin{cases}
d_1=1\,, &  \\
d_{n+1}=\min\left(d_n, \frac1{x_{n+1}}, \frac{n+1-\Phi(x_n)}{x_{n+1}-x_n}\right) & 
\end{cases}
$$
We notice that $\Phi$ is increasing and concave, because $\{d_{n}\}$ is positive and decreasing. In addition, since $d_{n+1}\leq \frac1{x_{n+1}}\leq \frac1x$ for $x\in (x_n, x_{n+1})$, we deduce that $\Phi'(x)\leq \frac1x$ for $x\geq 1$. Moreover, by definition of $d_{n+1}$,   we have
\be\label{phileq}
\Phi(x)\leq  \Phi(x_n) + \frac{n+1-\Phi(x_n)}{x_{n+1}-x_n}(x-x_{n}) \leq n+1 \quad \hbox{for $x\in(x_n, x_{n+1})$.}
\ee
This implies that
$$
\int_{1}^\infty \Phi(\lambda)\, f(\lambda)d\lambda  \leq \, \sup_k \, \sum_{n=1}^k (n+1)\int_{x_n}^{x_{n+1}}  f(\lambda)d\lambda
$$
where the last sum is finite, since reordering summation and using $x_n> a_n$ defined by \rife{an}, we have  
\begin{align*}
 \sum_{n=1}^k (n+1)\int_{x_n}^{x_{n+1}}  f(\lambda)d\lambda
  & = 2\int_{x_1}^\infty  f(\lambda)d\lambda - (k+1)\int_{x_{k+1}}^\infty  f(\lambda)d\lambda + \sum_{n=2}^k \int_{x_n}^{\infty}  f(\lambda)d\lambda
\\ & \leq 2\| f\|_{L^1}+ \sum_{n=2}^k \frac1{n^2} \leq C\,.
\end{align*}
Hence $\Phi\, f \in L^1(1,\infty)$. Finally, we check that $\Phi(r)\mathop{\to}\limits^{r\to \infty}\infty$; for that, being $\Phi$ increasing, it is enough to show that $\Phi(x_n)\to \infty$. But we have
\be\label{telesco}
\Phi(x_{n+1})= 1+ \sum_{k=1}^n (\Phi(x_{k+1})-\Phi(x_k))= 1+ \sum_{k=1}^n  d_{k+1}(x_{k+1}-x_k) 
\ee
If $d_{k+1}= \frac1{x_{k+1}}$, then $d_{k+1}(x_{k+1}-x_k) = 1- \frac{x_{k}}{x_{k+1}} \geq \frac12$ by definition of $x_k$. Similarly, if $d_{k+1}= \frac{k+1-\Phi(x_k)}{x_{k+1}-x_k}$, then $d_{k+1}(x_{k+1}-x_k) =k+1-\Phi(x_k)\geq 1$ (see \rife{phileq}). Therefore,  by definition of $d_k$, we conclude that $\sum_{k=1}^n  d_{k+1}(x_{k+1}-x_k) $  do not converge if $d_{k+1} \neq d_k$ for infinitely many indexes; but this means that the only case left out is that $d_k$ is definitively constant after some $k_0$, and again this implies $\sum_{k=1}^n  d_{k+1}(x_{k+1}-x_k)\geq d_{k_0} (x_{n+1}- x_{k_0}) \to \infty$. We conclude that the sum in \rife{telesco} diverges, hence $\Phi(x_n)\to \infty$.

Therefore, we proved that $\Phi$ is a concave, positive, piecewise linear function, which is increasing and unbounded, and satisfying $\Phi'(x)\leq \frac1x$ for $x\geq 1$; in addition,  $\Phi(\lambda) f(\lambda) \in L^1(0,\infty)$. Now, one can smoothen $\Phi$ by convolution with some function $\zeta\in C^\infty$ which is supported in $(-\frac12, \frac12)$ and with unit mass. Namely, we define
$$
\Psi'(r):= \int_{\R} \Phi(r-s)\zeta(s) ds, \qquad \Psi(r)= \int_0^r \Psi'(s)ds\,.
$$
We notice that $\Psi'\geq 0$, is increasing and differentiable, and satisfies $\Psi''(r) \leq \frac c r$. Moreover, we have
$\Psi'(r)\leq \Phi(r+\frac12) \leq \Phi(r) + \Phi(\frac12)$ (because $\Phi$ is concave and sublinear). Hence we also have $\Psi'(\lambda) f(\lambda) \in L^1(0,\infty)$. This implies $\Psi(m_0) \in L^1(M)$, since 
$$
\int_M \Psi(m_0) = \int_0^{\infty} \Psi'(\lambda) f(\lambda)d\lambda  <\infty\,.
$$
Notice also that $\Psi$ is superlinear, because $\Psi'(r)\to \infty$ as $r\to \infty$, by the very same property of $\Phi$.

{\bf Step 2.}  \quad Given $\Psi$ built in the previous step, which is a superlinear function, now we apply  \cite[Proposition 17.7]{Villani-oldnew} which provides with a function $U\leq \Psi$ such that $U$ is convex and superlinear, belongs to $C^2(0,\infty)$, and satisfies $ P'(r)r - (1-\frac1d) P(r) \geq 0$,  where $P(r)=U'(r)r-U(r)$. All those properties are proved in \cite[Proposition 17.7]{Villani-oldnew}. In addition,  we also wish to show that $P(r)\leq C\, r$, and to this purpose we recall the construction of $U$ in Villani's book. Given $\Psi(r)$, let us define $u(r)= r^{N}\Psi(r^{-N})$, and let $\tilde u(r)$ be the lower convex envelope of $u$ (the sup of linear functions lying below $u$). The function $U$ is then defined as
$$
U(r)= r \tilde u( r^{-1/N})\,.
$$
Since $\tilde u\leq u$, one obviously have $U(r)\leq \Psi(r)$ by definition. Now, let us define
$$
a_r:= \sup_{0<\xi\leq r^{-1/N}} \left[ \frac{u((2r)^{-1/N})-u(\xi)}\xi\right]
$$
which is a finite number because $u(\xi)\mathop{\to}\limits^{\xi \to 0}\infty $. Since $u(\xi) \geq u((2r)^{-1/N})-a_r \xi$ for $\xi\leq r^{-1/N}$, by definition of the convex envelope the same inequality is satisfied replacing $u$ with   $\tilde u$;  when $\xi= r^{-1/N}$, this yields
\be\label{ar}
 u((2r)^{-1/N})-\tilde u(r^{-1/N}) \leq a_r \, r^{-1/N}\,.
 \ee
 We estimate now $a_r$; in fact, we have
 $$
 a_r = \sup_{\lambda\in [1, \infty)}\left[ \frac{u((2r)^{-1/N})-u((\lambda r)^{-1/N})}{(\lambda r)^{-1/N}}\right]
 $$ 
 where we can write, by definition of $u$,
 \begin{align*}
 \frac{u((2r)^{-1/N})-u((\lambda r)^{-1/N})}{(\lambda r)^{-1/N}} & = (\lambda r)^{1/N}\left [ \frac{\Psi(2r)}{2r}-\frac{\Psi(\lambda r)}{\lambda r} \right] 
 \\
&  
= (\lambda r)^{1/N} \int_{\lambda}^2 \left( \frac{\Psi(s)}s\right)'_{\mathop{|}_{s=tr}} r \, dt\,.
 \end{align*}
Now, by the construction in Step 1, we know that $0\leq \Psi''(s) \leq \frac C s$ for some constant $C>0$. This implies that $\frac{\Psi(s)}s$ is increasing and we have
$$
\left( \frac{\Psi(s)}s\right)'= \frac{\Psi'(s)s-\Psi(s)}{s^2} \leq \frac C s
$$
by Lagrange theorem. Therefore, for any $\lambda \geq 1$ we estimate
$$
\int_{\lambda}^2 \left( \frac{\Psi(s)}s\right)'_{\mathop{|}_{s=tr}} r \, dt 
\leq C\, (\log 2) \, \1_{\{\lambda \in [1,2]\}}\,.
$$
  Coming back to $a_r$ we deduce that 
  $$
  a_r \leq C\, (\log2) (2r)^{1/N}
  $$
  and therefore, from \rife{ar}, we conclude that
\be\label{til}
  u((2r)^{-1/N})-\tilde u(r^{-1/N}) \leq \tilde C
 \ee
for some constant $\tilde C>0$. Now we can estimate $P(r)$; using the convexity of $U$ and  its definition, we have 
\begin{align*}
P(r) & = U'(r)r-U(r) \\ & \leq U(2r)-2 U(r) = 2r \tilde u ((2r)^{-1/N})-2r \tilde u(r^{-1/N})
\\ & \leq 2r (  u((2r)^{-1/N})-\tilde u(r^{-1/N}))
\end{align*}
where we used that $\tilde u\leq u$.  Using \rife{til} we conclude that $P(r) \leq 2 \tilde C \, r$. 
\end{proof}

{\bf Acknowledgments.} The authors are members of Gnampa research group of Indam. We acknowledge the support of the Excellence Project MatMod@TOV of the Department of Mathematics of the University of Rome Tor Vergata. We wish to warmly thank Giuseppe Savar\'e for fruitful discussions on the topics of the article.

\end{document}

%% file: config.tex
\usepackage[utf8]{inputenc}
\usepackage{lipsum}
\usepackage{amsmath}
\usepackage{amssymb}
\usepackage{amsthm}
\usepackage{mathtools}
\usepackage{braket}
\usepackage{relsize}
\usepackage{color,soul}
\usepackage{dsfont}
\usepackage{xcolor}
\usepackage{MnSymbol}
\usepackage{comment}
\usepackage{cite}
\usepackage{indentfirst}
\usepackage{microtype}
\usepackage{braket}
\usepackage[shortlabels]{enumitem}
\usepackage{verbatim}
\usepackage{listings}
\usepackage{color}
\usepackage{inconsolata}
\usepackage{tikz}
\usepackage{caption}
\usepackage{bm}
\usepackage{enumitem}
\usepackage{hyperref}
\sethlcolor{orange}

\newtheoremstyle{classicthm}{12pt}{12pt}{\itshape}{}{\scshape}{:}{1em}{}
\theoremstyle{classicthm}
\newtheorem{theorem}{Theorem}[section]
\theoremstyle{classicthm}
\newtheorem{lemma}[theorem]{Lemma}
\theoremstyle{classicthm}
\newtheorem{proposition}[theorem]{Proposition}
\newtheorem{definition}[theorem]{Definition}
\newtheorem{corollary}[theorem]{Corollary}
\newtheorem{remark}[theorem]{Remark}

\newcommand{\numberset}{\mathbb}
\newcommand{\N}{\numberset{N}}
\newcommand{\R}{\numberset{R}}

\newcommand{\1}{\mathds{1}}

\renewcommand{\epsilon}{\varepsilon}

\newcommand{\abs}[1]{\left|#1\right|}
\DeclarePairedDelimiter{\norma}{\lVert}{\rVert}

\DeclareMathOperator{\scalg}{\cdot_g\!}

%% file: log-entropy_manifold_arxiv.bbl
\begin{thebibliography}{99}


\bibitem{AGS} L. Ambrosio, N. Gigli, G. Savar\'e, {\it Gradient flows in metric spaces and in the space of probability measures}. Lectures in Mathematics ETH Z\"urich. Birkh\"auser Verlag, Basel, 2nd ed. (2008).



\bibitem{BB} J.-D. Benamou, Y. Brenier, {\it A computational fluid mechanics solution to the Monge- Kantorovich mass transfer problem},  Numer. Math., 84  (2000), 375-393.

\bibitem{BCS} J.-D. Benamou, , G. Carlier, F. Santambrogio, {\it Variational mean field games}.  Active Particles, Volume 1. Birkh\"auser, Cham, (2017),  141-171.

\bibitem{bochner}  M. Berger, P. Gauduchon, E.   Mazet: {\sl Le spectre d'une vari\'et\'e   riemannienne.} Lecture notes in mathematics, Springer 1971.

\bibitem{BJO}G. Buttazzo, C. Jimenez, E. Oudet, {\it  An optimization problem for mass transportation with congested
dynamics}, SIAM J. Control Optim., 48 (3) (2009), 1961-1976.


\bibitem{canizo} J.A. Canizo: {\sl The lemma of de la Vall\'ee-Poussin}, see  {https://canizo.org/tex/vallee-poussin.pdf}.

\bibitem{Carda1}P. Cardaliaguet, {\it Weak solutions for
  first order mean field games with local coupling}, in Analysis and geometry
  in control theory and its applications, Springer 2015, pp.~111--158.

\bibitem{CG} P. Cardaliaguet, P. J. Graber, {\it Mean field games systems of first order},  ESAIM Control Optim. Calc. Var. 21(2015), 690-722.

\bibitem{CGPT} P. Cardaliaguet, P. J. Graber, A. Porretta,  D. Tonon {\it Second order mean field games with degenerate diffusion and local coupling}, NoDEA Nonlinear Differential Equations Appl., 22 (2015), 1287-1317.

\bibitem{CMS} P. Cardaliaguet, A. R. M\'esz\'aros, F. Santambrogio, {\it First order mean field games with density constraints: pressure equals price}, SIAM J. Control Optim., 54 (2016),  2672-2709.

\bibitem{CMP} P. Cardaliaguet,  S. Munoz, A. Porretta, {\it Free boundary regularity and support propagation in mean field games and optimal transport}, preprint   arXiv:2308.00314. 

\bibitem{achdou2021mean} P. Cardaliaguet,  A. Porretta, {\it An introduction to Mean Field Game theory}, in: {Mean field games}, 1-158. Lecture Notes in Math., 2281 (CIME Found. Subser.),  Springer, Cham, (2020).

\bibitem{entropic} C. Clason, D.A. Lorenz, H. Mahler, B. Wirth, {\it Entropic regularization of continuous optimal transport problems}, 	 J. Math. Anal. Appl.  494 (2021), 124-432.



\bibitem{Erasquin-McCann} D. Cordero-Erausquin, R. J. McCann, M. Schmuckenschl\"ager, {\it A Riemannian interpolation inequality \`a  la Borell, Brascamp and Lieb}, Invent.  Math.  146 (2001), 219--257.   

\bibitem{Daneri_2008} S. Daneri, G. Savar\'e, {\it Eulerian Calculus for the Displacement Convexity in the Wasserstein Distance}, Siam J. Math. Anal. 40 (2008), 1104--1122.
 


\bibitem{FV} A. Figalli, C. Villani, {\it Strong displacement convexity on Riemannian manifolds}, Math. Z. 257 (2007), 251-259.


\bibitem{GiTa} N. Gigli, L. Tamanini, {\it Benamou-Brenier and duality formulas for the entropic cost on $RCD^*(K,N)$  spaces}, Probability Theory and Related Fields 176  (2020), 1-34.

\bibitem{gilbarg2001elliptic} D. Gilbarg, N. Trudinger. {\sl Elliptic Partial Differential Equations of Second Order}, 2nd ed. Springer, Berlin (1983). 
 
\bibitem{GoSe} D. Gomes, T. Seneci,  {\it Displacement convexity for first-order mean-field games}, Minimax Theory and   Appl. 3 (2018),  261-284.

\bibitem{GMST} P. J. Graber, A. R. M\'esz\'aros, F. Silva, D. Tonon, {\it The planning problem in mean field games as regularized mass transport},  Calc. Var. Partial Differential Equations 58 (2019).

\bibitem{grigoryan2009heat} A. Grigoryan, {\it Heat Kernel and Analysis on Manifolds}, AMA/IP Studies in Advances Mathematics, AMS 2009.

\bibitem{hebey} E. Hebey: {\sl Sobolev spaces on Riemannian manifolds}. Springer (2008).

 \bibitem{LSU} O. A.   Lady\v{z}enskaja, V.A. Solonnikov, N.N. Ural'ceva: {\sl Linear and quasi-linear equations of parabolic type}. Translations of Mathematical Monographs, Vol. 23 American Mathematical Society, Providence, R.I. 1967. 
 

 
\bibitem{LU} O. A.   Ladyzhenskaya, N. N.  Ural'tseva: {\sl Linear and quasilinear elliptic equations}. Academic Press, New York (1968).

\bibitem{LL1} J.-M. Lasry and P.-L. Lions, {\it Jeux \`a champ moyen. II-Horizon fini et contr$\hat{o}$le optimal}. Comptes Rendus Math\'ematique, 343
(2006), 679-684.

\bibitem{LL-japan} J.-M. Lasry and P.-L. Lions, {\em Mean field games},  Japanese journal of mathematics 2 (2007),  229-260.


\bibitem{LaSa} H. Lavenant, F. Santambrogio, {\it Optimal density evolution with congestion: $L^\infty$  bounds via flow interchange techniques and applications to variational Mean Field Games}, Comm. P.D.E.  43 (2018), 1761-1802.

 \bibitem{li_yau} P. Li, S.T. Yau, {\it On the parabolic kernel of the Schrödinger operator}, Acta Math. 156 (1986), 153 -- 201.
 
 \bibitem{Li-Zhang} L. Li, Z.L. Zhang, {\it On Li–Yau Heat Kernel Estimate}, Acta Math. Sinica (English Ser.) 37 (2021), 1205-1218.
 
 \bibitem{leonard} C. Leonard, {\it A survey of the Schrödinger problem and some of its connections with optimal transport},
Discrete Contin. Dyn. Syst., 34 (2014),  1533--1574.

\bibitem{lieb82} G. Lieberman, {\it  Solvability of quasilinear elliptic equations with nonlinear boundary conditions},  Trans. Amer. Math. Soc. 273 (1982), 753-765.

\bibitem{lieb84NA} G. Lieberman, {\it The nonlinear oblique derivative problem for quasilinear elliptic equations}, Nonlinear Anal. T.M.A. 8 (1984), 49-65.

\bibitem{lieb88} G. Lieberman, {\it Boundary regularity for solutions of degenerate elliptic equations}, Nonlinear
Anal. 12 (1988), 1203-1219.

\bibitem{LMS} M. Liero, A. Mielke, G. Savar\'e, {\it Optimal entropy-transport problems and a new Hellinger-Kantorovich distance between positive measures},  Invent. Math. 211 (2018),   969-1117. 

\bibitem{L-college} P.-L. Lions, {\it Cours at Coll\` ege de France (2009-2010)}, https://www.college-de-france.fr/site/pierre-louis-lions/course-2009-2010.htm.

\bibitem{LV} J. Lott, C. Villani, {\it Ricci curvature for metric-measure spaces via optimal transport}, Annals of Math. (2) 169 (2009), 903--991.

\bibitem{McCann} R. J. McCann, {\it A convexity principle for interacting gases},  Adv. Math.  128 (1997),  153-179.

\bibitem{McCann2} R.J. McCann, {\it Polar factorization of maps on Riemannian manifolds}, Geom. Funct. Anal. 11 (2001), 589--608.

\bibitem{munoz1} N. Mimikos-Stamatopoulos, S. Munoz, {\it Regularity of one-dimensional first-order mean field games and the planning problem}, preprint arXiv:2204.06474 (2022).

\bibitem{Mu} S. Munoz,  {\it Classical and weak solutions to local first order mean field games through elliptic regularity}, Ann. I. H. Poincar\'e Anal. Nonlin\'eaire 39 (2022), 1-39.


\bibitem{ORRIERI20191868} C. Orrieri, A. Porretta, G. Savar\'e, {\it A variational approach to the mean field planning problem}, J. Funct. Anal. 277 (2019),  1868-1957.

\bibitem{otto-villani} F. Otto, C. Villani, {\it Generalization of an Inequality by Talagrand and Links with the Logarithmic Sobolev Inequality},  J. Funct. Anal. 277 (2019), 361-400.


\bibitem{Po} A. Porretta, {\it Regularizing effects of the entropy functional in optimal transport and planning problems}, J.   Funct.  Anal.  284 (2023) 109-759. 


\bibitem{Sa} F. Santambrogio: {\sl Optimal transport for applied mathematicians}.
 Progress in Nonlinear Differential Equations and their applications 87, Birkh\"auser, 2015.
 
%
%

\bibitem{Simon} J. Simon,  {\it Compact sets in $L^p(0,T;B)$},  Ann. Mat. Pura Appl.  146 (4)
(1987),   65-96.

 
\bibitem{taylor2010partial} M.E. Taylor: {\sl Partial Differential Equations I: Basic Theory}.  Applied Mathematical Sciences (2010), Springer New York.

\bibitem{Renesse} M.-K. von  Renesse, K.-T. Sturm, {\it Transport inequalities, gradient estimates, entropy, and Ricci curvature},  Comm. Pure Appl. Math. 58 (2005), 923-940.


\bibitem{Villani_topics} C. Villani: {\sl Topics in Optimal transportation}.  Graduate Studies in Mathematics, vol. 58. American
Mathematical Society, Providence, RI (2003).

\bibitem{Villani-oldnew} C. Villani: {\sl  Optimal transport: Old and New}. Grundlehren der mathematischen Wissenschaften 338, Springer-Verlag and Heidelberg GmbH \& Co. KG (2016).

\end{thebibliography}
